\documentclass[10pt]{article}
\usepackage{amsmath}
\usepackage{amsthm}
\usepackage{amssymb}
\usepackage{enumitem}

\usepackage{tikz}
\usetikzlibrary{arrows, automata, decorations.pathreplacing, fit, matrix, patterns, positioning, calc}
\usepackage{tikz-qtree}
\usepackage{xifthen} 

\usepackage[small,it,singlelinecheck=false]{caption}

\usepackage{titlefoot}
\usepackage[small]{titlesec}
\usepackage{units} 
\usepackage[square,sort,comma,numbers]{natbib}
\usepackage{verbatim}

\usepackage{color}
\definecolor{lightgray}{rgb}{0.8, 0.8, 0.8}
\definecolor{darkgray}{rgb}{0.65, 0.65, 0.65}

\usepackage[bookmarks]{hyperref}
\hypersetup{
        colorlinks=true,
        linkcolor=black,
        anchorcolor=black,
        citecolor=black,
        urlcolor=black,
        pdfpagemode=UseThumbs,
        pdftitle={Rationality For Subclasses of 321-Avoiding Permutations},
        pdfsubject={Combinatorics},
        pdfauthor={Albert, Brignall, Ruskuc, and Vatter},
}

\newcounter{todocounter}

\theoremstyle{plain}
\newtheorem{theorem}{Theorem}[section]
\newtheorem{observation}[theorem]{Observation}
\newtheorem{proposition}[theorem]{Proposition}

\newtheorem{corollary}[theorem]{Corollary}
\newtheorem{conjecture}[theorem]{Conjecture}

\theoremstyle{definition}

\makeatletter
\newfont{\footsc}{cmcsc10 at 8truept}
\newfont{\footbf}{cmbx10 at 8truept}
\newfont{\footrm}{cmr10 at 10truept}
\renewenvironment{abstract}%
                {
                  \begin{list}{}%
                     {\setlength{\rightmargin}{1in}%
                      \setlength{\leftmargin}{1in}}%
                   \item[]\ignorespaces\begin{small}}%
                 {\end{small}\unskip\end{list}}

\newcommand{\Av}{\operatorname{Av}}

\newcommand{\C}{\mathcal{C}}
\newcommand{\D}{\mathcal{D}}

\newcommand{\G}{\mathcal{G}}

\renewcommand{\L}{\mathcal{L}}

\renewcommand{\P}{\mathcal{P}}
\newcommand{\R}{\mathcal{R}}

\newcommand{\W}{\mathcal{W}}

\newcommand{\st}{\::\:}
\newcommand{\leftcount}{\operatorname{left}}
\newcommand{\rightcount}{\operatorname{right}}
\newcommand{\bij}{\varphi}

\newcommand{\gridded}{\sharp}

\newcommand{\oldpoint}{\mathord{\circ}}
\newcommand{\newpoint}{\mathord{\bullet}}
\newcommand{\first}{^{\text{\scriptsize st}}}
\newcommand{\nd}{^{\text{\scriptsize nd}}}
\newcommand{\rd}{^{\text{\scriptsize rd}}}
\renewcommand{\th}{^{\text{\scriptsize th}}}

\makeatletter
\let\start@align@nopar\start@align
\let\start@gather@nopar\start@gather
\let\start@multline@nopar\start@multline
\long\def\start@align{\par\start@align@nopar}
\long\def\start@gather{\par\start@gather@nopar}
\long\def\start@multline{\par\start@multline@nopar}
\makeatother

\newcommand\absdot[2]{
	\node at #1 {\small $\bullet$};
	\node at #1 [below] {$#2$};
}

\newcommand\absdothollow[2]{
	\node at #1 {\small \textcolor{white}{$\bullet$}};
	\node at #1 {\small $\circ$};
	\node at #1 [below] {$#2$};
}

\newcommand{\plotperm}[1]{
	\foreach \j [count=\i] in {#1} {
		\absdot{(\i,\j)}{};
	};
}

\newcommand{\plotpartialperm}[1]{
	\foreach \i/\j in {#1} {
		\absdot{(\i,\j)}{};
	};
}

\newcommand{\plotpartialpermhollow}[1]{
	\foreach \i/\j in {#1} {
		\absdothollow{(\i,\j)}{};
	};
}

\newcommand{\plotpermbox}[4]{
	\draw [darkgray, thick, rounded corners=0.01, line cap=round]
		({#1-0.5}, {#2-0.5}) rectangle ({#3+0.5}, {#4+0.5});
}

\newcommand{\plotpermgraph}[1]{
	\foreach \j [count=\i] in {#1} {
		\foreach \b [count=\a] in {#1} {
			\ifthenelse{\a<\i \AND \b>\j}{\draw (\a,\b)--(\i,\j);}{}
		};
	};
	\plotperm{#1};
}

\newcommand{\plotpermdyckgrid}[1]{
	\foreach \i [count=\n] in {#1} {};
	\foreach \i in {0,1,2,...,\n}{
		\draw [color=lightgray] ({\i+0.5}, 0.5)--({\i+0.5}, {\n+0.5});
		\draw [color=lightgray] (0.5, {\i+0.5})--({\n+0.5}, {\i+0.5});
	}
	\draw [very thick, rounded corners=0.01, line cap=round, color=darkgray] (0.5, 0.5)--({\n+0.5}, {\n+0.5});
	\plotperm{#1};
}

\newcommand{\plotpermdyckpath}[1]{
	\draw [very thick, rounded corners=0.01, line cap=round] (0.5,0.5)
	\foreach \step in {#1} {
		\ifnum\step=1
			-- ++(0,1)
		\else
			-- ++(1,0)
		\fi
	};
}

\newcommand{\plotdyckpath}[1]{
	\draw [very thick, rounded corners=0.01, line cap=round] (0.5,0)
	\foreach \step in {#1} {
		\ifnum\step=1
			-- ++(1,1)
		\else
			-- ++(1,-1)
		\fi
	};
}

\newcommand{\etaright}[1]
	{
		\begin{tikzpicture}[scale=.2, anchor=base, baseline]
			\node[inner sep=0pt] at (0,0){$#1$};
			\useasboundingbox (current bounding box.south west) rectangle (current bounding box.north east);
			\draw [thin] (-0.3,1.4)--(0.1,1.4); 
			\draw [fill, ultra thin] (0.4,1.4)--(0.1,1.25)--(0.1,1.55)--(0.4,1.4); 
		\end{tikzpicture}
	}
\newcommand{\etaleft}[1]
	{
		\begin{tikzpicture}[scale=.2, anchor=base, baseline]
			\node[inner sep=0pt] at (0,0){$#1$};
			\useasboundingbox (current bounding box.south west) rectangle (current bounding box.north east);
			\draw [thin] (-0.2,1.4)--(0.3,1.4); 
			\draw [fill, ultra thin] (-0.5,1.4)--(-0.2,1.25)--(-0.2,1.55)--(-0.5,1.4); 
 		\end{tikzpicture}
	}
\newcommand{\etaleftright}[1]
	{
		\begin{tikzpicture}[scale=.2, anchor=base, baseline]
			\node[inner sep=0pt] at (0,0){$1$};
			\useasboundingbox (current bounding box.south west) rectangle (current bounding box.north east);
			\draw [thin] (-0.2,1.4)--(0.1,1.4); 
			\draw [fill, ultra thin] (-0.5,1.4)--(-0.2,1.25)--(-0.2,1.55)--(-0.5,1.4); 
			\draw [fill, ultra thin] (0.4,1.4)--(0.1,1.25)--(0.1,1.55)--(0.4,1.4); 
		\end{tikzpicture}
	}

\newcommand{\fnmatrix}[2]{\text{$\left(\text{\begin{footnotesize}$\begin{array}{#1}#2\end{array}$\end{footnotesize}}\right)$}}
\newcommand{\fnarray}[2]{\text{$\text{\begin{footnotesize}$\begin{array}{#1}#2\end{array}$\end{footnotesize}}$}}

\newcommand{\eval}[2][\right]{\relax\ifx#1\right\relax \left.\fi#2#1\rvert}

\datefoot{\today}
\amssubj{05A05, 05A15}

\newpagestyle{main}[\small]{
        \headrule
        \sethead[\usepage][][]
        {\sc Rationality For Subclasses of $321$-Avoiding Permutations}{}{\usepage}}

\title{\sc Rationality For Subclasses of $321$-Avoiding Permutations}
\author{
	\begin{tabular}{cc}
        Michael Albert\footnote{Albert, Ru\v{s}kuc, and Vatter were partially supported by EPSRC via the grant EP/J006440/1.}&Robert Brignall\\
		{\small Department of Computer Science}&{\small Department of Mathematics and Statistics}\\[-3pt]
		{\small University of Otago}&{\small The Open University}\\[-3pt]
		{\small Dunedin, New Zealand}&{\small Milton Keynes, England UK}\\[20pt]
        Nik Ru\v{s}kuc\footnotemark[\value{footnote}]&Vincent Vatter\footnotemark[\value{footnote}]\footnote{Vatter's research was partially supported by the National Science Foundation under Grant Number DMS-1301692 and the National Security Agency under Grant Number H98230-16-1-0324. The United States Government is authorized to reproduce and distribute reprints not-withstanding any copyright notation herein.}\\
		{\small School of Mathematics and Statistics}&{\small Department of Mathematics}\\[-3pt]
		{\small University of St Andrews}&{\small University of Florida}\\[-3pt]
		{\small St Andrews, Scotland UK}&{\small Gainesville, Florida USA}\\[20pt]
	\end{tabular}
}

\titleformat{\section}
        {\large\sc}
        {\thesection.}{1em}{}

\date{}

\begin{document}
\maketitle


\pagestyle{main}

\begin{abstract}
We prove that every proper subclass of the $321$-avoiding permutations that is defined either by only finitely many additional restrictions or is well-quasi-ordered has a rational generating function. To do so we show that any such class is in bijective correspondence with a regular language. The proof makes significant use of formal languages and of a host of encodings, including a new mapping called the panel encoding that maps languages over the infinite alphabet of positive integers avoiding certain subwords to languages over finite alphabets.
\end{abstract}

\section{Introduction}

It has been known since 1968, when the first volume of Knuth's \emph{The Art of Computer Programming}~\cite{knuth:the-art-of-comp:1} was published, that the $312$-avoiding permutations and the $321$-avoiding permutations are both enumerated by the Catalan numbers, and thus have algebraic generating functions.  At least nine essentially different bijections between these two permutation classes have been devised in the intervening years, as surveyed by Claesson and Kitaev~\cite{claesson:classification-:}. In one such bijection (shown in Figure~\ref{fig-bij-Dyck} and first given in this non-recursive form by Krattenthaler~\cite{krattenthaler:permutations-wi:}) we obtain Dyck paths from permutations of both types by drawing a path above their left-to-right maxima (an entry is a \emph{left-to-right maximum} if it is greater than every entry to its left).

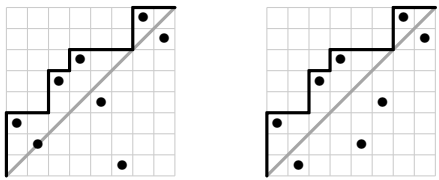
\begin{figure}
\begin{footnotesize}
\begin{center}
\begin{tabular}{ccc}
	\begin{tikzpicture}[scale=.28]
	    \plotpermdyckgrid{3,2,5,6,4,1,8,7};
		\plotpermdyckpath{1,1,1,-1,-1,1,1,-1,1,-1,-1,-1,1,1,-1,-1};
    \end{tikzpicture} 
&\quad\quad&
    \begin{tikzpicture}[scale=.28]
	    \plotpermdyckgrid{3,1,5,6,2,4,8,7};
		\plotpermdyckpath{1,1,1,-1,-1,1,1,-1,1,-1,-1,-1,1,1,-1,-1};
    \end{tikzpicture} 
\end{tabular}
\end{center}
\end{footnotesize}
\caption{The bijections to Dyck paths from $312$-avoiding permutations (left) and $321$-avoiding permutations (right). Knowing the positions and values of the left to right maxima, the remaining elements can be added in a unique fashion to avoid $312$, respectively $321$.
}
\label{fig-bij-Dyck}
\end{figure}

\begin{figure}
\begin{center}
	\begin{tikzpicture}[scale=1]
		\node (1) at (0,0) {$1$};
		\node (12) at (-1,1) {$12$};
		\node (21) at (1,1) {$21$};
		\node (123) at (-2,2) {$123$};
		\node (132) at (-1,2) {$132$};
		\node (213) at (0,2) {$213$};
		\node (231) at (1,2) {$231$};
		\node (321) at (2,2) {$321$};
		\draw [thick, line cap=round]  (1)--(12)--(1)--(21);
		\draw [thick, line cap=round] (12)--(123)--(12)--(132)--(12)--(213)--(12)--(231)--(12);
		\draw [thick, line cap=round]  (21)--(132)--(21)--(213)--(21)--(231)--(21)--(321);
	\end{tikzpicture}
\quad\quad
	\begin{tikzpicture}[scale=1]
		\node (1) at (0,0) {$1$};
		\node (12) at (-1,1) {$12$};
		\node (21) at (1,1) {$21$};
		\node (123) at (-2,2) {$123$};
		\node (132) at (-1,2) {$132$};
		\node (213) at (0,2) {$213$};
		\node (231) at (1,2) {$231$};
		\node (312) at (2,2) {$312$};
		\draw [thick, line cap=round]  (1)--(12)--(1)--(21);
		\draw [thick, line cap=round]  (12)--(123)--(12)--(132)--(12)--(213)--(12)--(231)--(12)--(312);
		\draw [thick, line cap=round]  (21)--(132)--(21)--(213)--(21)--(231)--(21)--(312);
	\end{tikzpicture}
\end{center}
\caption{The Hasse diagrams of the $312$-avoiding (left) and $321$-avoiding (right) permutations.}
\label{fig-Hasse}
\end{figure}
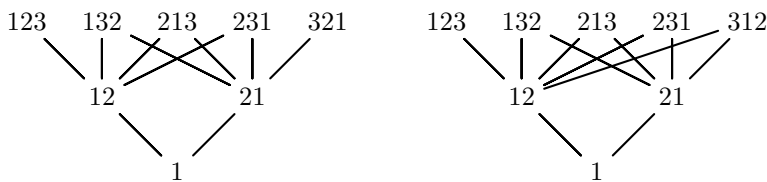

Despite their equinumerosity, there are fundamental differences between these two classes. Indeed, Miner and Pak~\cite{miner:the-shape-of-ra:} make a compelling argument that there are so many different bijections between these two classes precisely \emph{because} they are so different, and thus there can be no ``ultimate'' bijection. In particular, both sets carry a natural ordering with respect to the containment of permutations (defined below) but they \emph{are not} isomorphic as partially ordered sets. Indeed, this can be seen by examining the first three levels of their Hasse diagrams, drawn in Figure~\ref{fig-Hasse}.

A more striking difference between the two classes is that the $321$-avoiding permutations contain infinite antichains (see Section~\ref{sec-wpo}), while the $312$-avoiding permutations do not. Following the standard terminology, we say that a permutation class without infinite antichains is  \emph{well-quasi-ordered}.

From a structural perspective, the avoidance of $312$ imposes severe restrictions on permutations: the entries to the left of the minimum must lie below the entries to the right of this minimum. This restricted structure is known to imply that proper subclasses of the $312$-avoiding permutations are very well-behaved: there are only countably many such subclasses, and as Albert and Atkinson~\cite{albert:simple-permutat:} proved in their work on the substitution decomposition, each has a rational generating function. (Mansour and Vainshtein~\cite{mansour:restricted-132-:2001} had proved this rationality result for proper subclasses classes defined by a single additional restriction earlier.)

The $321$-avoiding permutations also have a good deal of structure: their entries can be partitioned into two increasing subsequences. However, this property has proved much more difficult to work with.  In particular, as noted above, there are infinite antichains of $321$-avoiding permutations, so there are uncountably many proper subclasses of this class---in fact uncountably many subclasses with pairwise distinct generating functions.  By an elementary counting argument, \emph{some} of these proper subclasses must have non-rational (indeed, also non-algebraic and non-D-finite) generating functions. 


Because $\Av(321)$ is not well-quasi-ordered, any result analogous to the one mentioned for $312$-avoiding permutations (which are, to repeat, well-quasi-ordered) must be more discerning as to the subclasses considered. We develop a methodology for working with arbitrary subclasses of $\Av(321)$ and show how to apply it to two natural general families:  subclasses defined by imposing finitely many additional forbidden patterns and subclasses that are well-quasi-ordered. Our main result shows that either of these conditions is sufficient to guarantee the rationality of generating functions.  

For the rest of the introduction, we review the formal definitions of permutation containment and permutation classes. We generally represent permutations in one line notation as sequences of positive integers. We define the \emph{length} of the permutation $\pi$, denoted $|\pi|$, to be the length of the corresponding sequence, i.e., the cardinality of the domain of $\pi$. Given permutations $\pi$ and $\sigma$, we say that $\pi$ \emph{contains} $\sigma$, and write $\sigma\le\pi$, if $\pi$ has a subsequence $\pi(i_1)\cdots\pi(i_{|\sigma|})$ of the same length as $\sigma$ that is \emph{order isomorphic} to $\sigma$ (i.e., $\pi(i_s) < \pi(i_t)$ if and only if $\sigma(s) < \sigma(t)$ for all $1\le s,t\le |\sigma|$); otherwise, we say that $\pi$ \emph{avoids} $\sigma$.  If $\pi$ contains $\sigma$ we also say that $\sigma$ is a \emph{subpermutation} of $\pi$ particularly in contexts where we have a specific embedding (i.e., set of indices) in mind. Containment is a partial order on permutations. For example, $\pi=251634$ contains $\sigma=4123$, as can be seen by considering the subsequence $\pi(2)\pi(3)\pi(5)\pi(6)=5134$.  A collection of permutations $\C$ is a \emph{permutation class} if it is closed downwards in this order; i.e., if $\pi\in\C$ and $\sigma\le\pi$, then $\sigma\in\C$. 

For any permutation class $\C$ there is a unique antichain $B$ (in the set of all permutations) such that
\[
\C=\Av(B)=\{\pi\st \pi\text{ avoids all } \beta \in B\}.
\]
This antichain, consisting of the minimal permutations \emph{not} in $\C$, is called the \emph{basis} of $\C$.  If $B$ happens to be finite, we say that $\C$ is \emph{finitely based}.  For non-negative integers $n$, we denote by $\C_n$ the set of permutations in $\C$ of length $n$, and refer to
\[
\sum_{n} |\C_n|x^n=\sum_{\pi\in\C} x^{|\pi|}
\]
as the \emph{generating function} of $\C$. The goal of this paper is to establish the following.

\begin{theorem}
\label{thm-321-rational}
If a proper subclass of the $321$-avoiding permutations is finitely based or well-quasi-ordered then it has a rational generating function.
\end{theorem}

In~\cite{bousquet-melou:rational-and-al:} Bousquet-M\'elou writes
\begin{quote}
``for almost all families of combinatorial objects with a rational [generating function], it is easy to foresee that there will be a bijection between these objects and words of a regular language''.
\end{quote}
In proving Theorem~\ref{thm-321-rational} we indeed adopt an approach via regular languages. We in fact encode permutations as words using several different encodings. We begin by introducing the \emph{domino encoding} that records the relative positions of entries in pairs of adjacent cells in a staircase gridding. After that we combine this information and encode each $321$-avoiding permutation as a word, say $w$, over the positive integers $\mathbb{P}$ satisfying the additional condition $w(i+1)\le w(i)+1$ for all relevant indices $i$ (throughout we denote by $w(i)$ the $i\th$ letter of the word $w$). We then show that for any proper subclass, $\C$, of $321$-avoiding permutations there is some positive integer $c$ such that the encoding of every permutation in $\C$ avoids (as a subword) every shift of the word $(12\cdots c)^c$, i.e. all words $(i(i+1)\cdots (i+c-1))^c$ for $i\in\mathbb{P}$. The true key to our method is the \emph{panel encoding} $\eta_c$, which translates languages not containing shifts of $(12\cdots c)^c$ to languages over \emph{finite} alphabets. A careful analysis of the interplay between panel encodings, domino encodings, and the classical encodings by Dyck paths (from Figure~\ref{fig-bij-Dyck}) along with a technique called marking establishes the regularity of various images under $\eta_c$, and this completes the proof of Theorem~\ref{thm-321-rational}.

We assume throughout that the reader has some familiarity with the basics of regular languages, as provided by Sakarovitch~\cite{Sakarovitch:Elements-of-aut:}; for a more combinatorial approach we refer the reader to Bousquet-M\'elou~\cite{bousquet-melou:rational-and-al:} or Flajolet and Sedgewick~\cite[Section I.4 and Appendix A.7]{flajolet:analytic-combin:}. The notation used is mostly standard. Herein a \emph{subword} of the word $w$ is any subsequence of its entries while a \emph{factor} is a contiguous subsequence. Given a set of letters $X$ and a word $w$ we denote by $\eval{w}^X$ the \emph{projection} of $w$ onto $X$, i.e., the subword of $w$ formed by its letters in $X$. Finally, we denote the empty word by $\epsilon$.

\section{Staircase Griddings}
\label{sec-staircase}


A \emph{staircase gridding} of a $321$-avoiding permutation $\pi$ is a partition of its entries into \emph{cells} labelled by the positive integers satisfying four properties:
\begin{itemize}
\item the entries in each cell are increasing,
\item for $i\ge 1$, all entries in the $(2i)\th$ cell lie to the right of those in the $(2i-1)\first$ cell,
\item for $i\ge 1$, all entries in the $(2i+1)\first$ cell lie above those in the $(2i)\th$ cell, and
\item if $j\ge i+2$ then all entries in the $j\th$ cell lie above and to the right of those in the $i\th$ cell.
\end{itemize}
Staircase griddings have been used extensively in the study of $321$-avoiding permutations, for instance in ~\cite{albert:growth-rates-fo:,albert:generating-and-:,albert:the-complexity-:,guillemot:pattern-matchin:} and represent the fundamental objects of consideration here. We denote by $\pi^\gridded$ a particular staircase gridding of the $321$-avoiding permutation $\pi$.

Every $321$-avoiding permutation has at least one staircase gridding and indeed, we can identify a preferred staircase gridding of every such permutation: a staircase gridding of the $321$-avoiding permutation $\pi$ is \emph{greedy} if the first cell contains as many entries as possible, and subject to this, the second cell contains as many entries as possible, and so on. Figure~\ref{fig-example-permutations-decomposition} provides an example of a greedy staircase gridding. 

It is easy to construct greedy staircase griddings in the following iterative manner. The entries of the first cell are the maximum increasing prefix $\tau$ of $\pi$. Those of the second cell are then the maximum increasing sequence in $\pi\setminus\tau$ whose values form an initial segment of the values occurring in $\pi \setminus \tau$. Thereafter we continue alternately taking a maximum increasing prefix and then a maximum increasing sequence of values forming an initial segment of the values remaining.

\begin{figure}
\begin{center}
    \begin{tikzpicture}[scale=.25]
		\plotperm{2,3,1,4,7,8,5,11,6,9,12,10,14,13,15};
		\plotpermbox{1}{1}{2}{6}; 
		\plotpermbox{3}{1}{9}{6}; 
		\plotpermbox{3}{7}{9}{11}; 
		\plotpermbox{10}{7}{13}{11}; 
		\plotpermbox{10}{12}{13}{15}; 
		\plotpermbox{14}{12}{15}{15}; 
		\node [below right] at ({0+0.25},{6+0.65}) {\scriptsize $1$};
		\node [below right] at ({2+0.25},{11+0.65}) {\scriptsize $3$};
		\node [below right] at ({9+0.25},{15+0.65}) {\scriptsize $5$};
		\node [above left] at ({10-0.25},{0+0.35}) {\scriptsize $2$};
		\node [above left] at ({14-0.25},{6+0.35}) {\scriptsize $4$};
		\node [above left] at ({16-0.25},{11+0.35}) {\scriptsize $6$};
    \end{tikzpicture} 
\end{center}
\caption{The greedy staircase gridding of the $321$-avoiding permutation $2\ 3\ 1\ 4\ 7\ 8\ 5\ 11\ 6\ 9\ 12\ 10\ 14\ 13\ 15$.}
\label{fig-example-permutations-decomposition}
\end{figure}
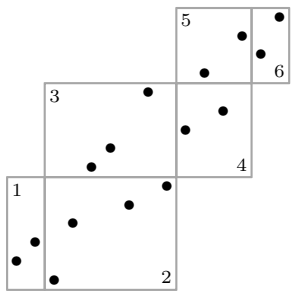

The relative position of two entries in a $321$-avoiding permutation $\pi$ is completely determined by the numbers given to their cells in any staircase gridding, unless these numbers are consecutive. In the case of cells which lie next to each other horizontally we consider their entries as being \emph{ordered from bottom to top}, and in the case of cells which lie next to each other vertically, from \emph{left to right}. Observe that this gives us two orders on the entries of a given cell (except the first), but that the two orders in fact coincide. With this ordering in mind, we formulate two conditions that characterise greedy staircase griddings:
\begin{enumerate}[label=\textup{(G\arabic*)}, leftmargin=*, widest=1]
\item For all $i\ge 1$ the first entry in the $(i+1)\first$ cell occurs before the first entry of the $(i+2)\nd$ cell.
\item For all $i\ge 1$ the first entry in the $(i+1)\first$ cell occurs before the last entry of the $i\th$ cell.
\end{enumerate}
These restrictions, or rather how they can fail, are depicted in Figure~\ref{fig-G1-G2}. It is important for later to note that these conditions can be tested by inspecting only the first and last entries of each cell.

\begin{figure}
\begin{center}
    \begin{tikzpicture}[scale=.25]
		\draw [darkgray, thick, pattern=north west lines, pattern color=darkgray] (0,-4)--(0,4)--(2,4)--(2,2)--(4,2)--(4,0)--(2,0)--(2,-4)--cycle;
		\draw [darkgray, very thick, rounded corners=0.01, line cap=round] (-4,-6)--(-4,0)--(6,0);
		\draw [darkgray, very thick, rounded corners=0.01, line cap=round] (0,-6)--(0,4)--(6,4);
		\draw [darkgray, very thick, rounded corners=0.01, line cap=round] (4,-4)--(4,6);
		\draw [darkgray, very thick, rounded corners=0.01, line cap=round] (-4,-4)--(4,-4);
		\absdot{(2,2)}{};
		\node [above left] at (4,0) {$i+2$};
    \end{tikzpicture} 
\quad\quad
    \begin{tikzpicture}[scale=.25]
		\draw [darkgray, thick, pattern=north west lines, pattern color=darkgray] (-4,0)--(-4,2)--(0,2)--(0,4)--(2,4)--(2,2)--(4,2)--(4,0)--cycle;
		\draw [darkgray, very thick, rounded corners=0.01, line cap=round] (-6,-4)--(-4,-4)--(-4,4)--(6,4);
		\draw [darkgray, very thick, rounded corners=0.01, line cap=round] (0,-4)--(0,6);
		\draw [darkgray, very thick, rounded corners=0.01, line cap=round] (-6,0)--(4,0)--(4,6);
		\draw [darkgray, very thick, rounded corners=0.01, line cap=round] (-4,-4)--(0,-4);
		\absdot{(2,2)}{};
		\node [above left] at (4,0) {$i+2$};
    \end{tikzpicture}
\quad\quad
    \begin{tikzpicture}[scale=.25]
		\draw [darkgray, thick, pattern=north west lines, pattern color=darkgray] (-4,-2)--(-4,0)--(0,0)--(0,4)--(2,4)--(2,-2)--(4,-2)--(4,-4)--(0,-4)--(0,-2)--cycle;
		\draw [darkgray, very thick, rounded corners=0.01, line cap=round] (-4,-6)--(-4,0)--(6,0);
		\draw [darkgray, very thick, rounded corners=0.01, line cap=round] (0,-6)--(0,4)--(6,4);
		\draw [darkgray, very thick, rounded corners=0.01, line cap=round] (4,-4)--(4,6);
		\draw [darkgray, very thick, rounded corners=0.01, line cap=round] (-4,-4)--(4,-4);
		\absdot{(2,-2)}{};
		\node [above left] at (4,-4) {$i+1$};
    \end{tikzpicture} 
\quad\quad
    \begin{tikzpicture}[scale=.25]
		\draw [darkgray, thick, pattern=north west lines, pattern color=darkgray] (-4,0)--(-4,4)--(-2,4)--(-2,2)--(4,2)--(4,0)--(0,0)--(0,-4)--(-2,-4)--(-2,0)--cycle;
		\draw [darkgray, very thick, rounded corners=0.01, line cap=round] (-6,-4)--(-4,-4)--(-4,4)--(6,4);
		\draw [darkgray, very thick, rounded corners=0.01, line cap=round] (0,-4)--(0,6);
		\draw [darkgray, very thick, rounded corners=0.01, line cap=round] (-6,0)--(4,0)--(4,6);
		\draw [darkgray, very thick, rounded corners=0.01, line cap=round] (-4,-4)--(0,-4);
		\absdot{(-2,2)}{};
		\node [above left] at (0,0) {$i+1$};
    \end{tikzpicture}
\end{center}
\caption{The four types (due to parity) of failures of (G1) and (G2). Here the hatched regions indicate positions where entries do not lie. Within the cell of the indicated entry these serve to identify it as the first entry of its cell. In the two rightmost pictures the hatched region in cell $i+2$ is empty because the gridding is assumed to satisfy (G1).}
\label{fig-G1-G2}
\end{figure}
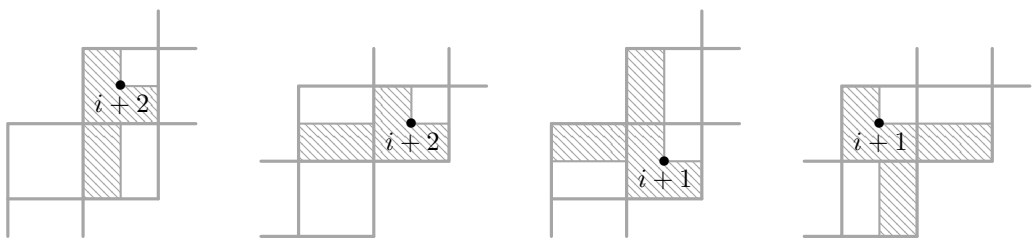

\begin{proposition}
\label{prop-greedy-G1-G2}
A staircase gridding is greedy if and only if it satisfies {\normalfont (G1)} and {\normalfont (G2)}.
\end{proposition}
\begin{proof}
Let $\pi$ be a 321-avoiding permutation, and consider first its greedy staircase gridding. If this gridding were to fail (G1) for some $i\ge 1$, then we see from the two leftmost pictures in Figure~\ref{fig-G1-G2} that the first entry of the $(i+2)\nd$ cell could (and therefore, in a greedy gridding, would) have been placed instead in the $i\th$ cell, a contradiction. On the other hand, if the gridding were to satisfy (G1) but fail (G2) for some $i\ge 1$ then we see from the two rightmost pictures in Figure~\ref{fig-G1-G2} that the first entry of the $(i+1)\first$ cell would have been placed in the $i\th$ cell, another contradiction.

Next consider a staircase gridding $\pi^\gridded$ of $\pi$ that satisfies (G1) and (G2). The condition (G2) implies that the labels of the non-empty cells form an initial segment of $\mathbb{P}$ so we proceed inductively. By definition, the entries of the $1\first$ cell form an initial increasing segment of $\pi$ so we need to show that it is the longest such segment. The next entry of $\pi$ (reading left to right) must lie in the $2\nd$ cell because (G1) shows that the leftmost entry of the $2\nd$ cell lies to the left of all entries of the $3\rd$ cell. Thus this entry is the first entry of the $2\nd$ cell. By (G2) it lies below an entry of the $1\first$ cell, and this implies that the entries of the $1\first$ cell are a maximum initial increasing segment.

Let $\tau$ denote the contents of the $1\first$ cell and consider the entries of the $2\nd$ cell of $\pi$. By the third and fourth requirements for a staircase gridding, all entries of $\pi$ not belonging to the first or second cells lie above those in the second cell. Thus the entries of the second cell form an increasing contiguous sequence by value in $\pi\setminus\tau$ and we must show that it is maximum. Consider the next smallest entry of $\pi\setminus\tau$ by value (if there is no such entry then we are done). As before, (G1) shows that this entry must lie in the $3\rd$ cell, and thus must be the least entry of the $3\rd$ cell. Again, (G2) implies that this entry lies to the left of an entry of the $2\nd$ cell, and thus the contents of the $2\nd$ cell are maximum.

To complete the argument we repeat the reasoning for the $1\first$ and $2\nd$ cells for odd cells and even cells respectively, with suitable modifications, basically referring throughout to the set of entries of $\pi$ that belong to the remaining cells of $\pi^\gridded$.
\end{proof}

Staircase griddings have a pleasing geometric interpretation, as first observed by Waton in his thesis~\cite{waton:on-permutation-:}. First we describe a general construction: given any figure in the plane and permutation $\pi$ we say that $\pi$ can be \emph{drawn} on the figure if we can choose a set $P$ consisting of $n$ points in the figure, no two on a common horizontal or vertical line, label them $1$ to $n$ from bottom to top and then read them from left to right to obtain $\pi$. If this relationship holds between $P$ and $\pi$ we say that $P$ and $\pi$ are \emph{order isomorphic}.

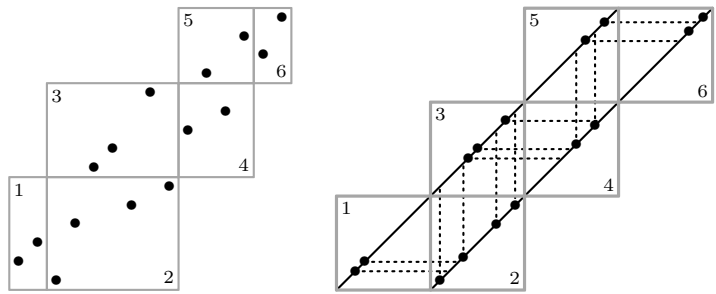
\begin{figure}
\begin{center}
    \begin{tikzpicture}[scale=.25]
		\plotperm{2,3,1,4,7,8,5,11,6,9,12,10,14,13,15};
		\plotpermbox{1}{1}{2}{6}; 
		\plotpermbox{3}{1}{9}{6}; 
		\plotpermbox{3}{7}{9}{11}; 
		\plotpermbox{10}{7}{13}{11}; 
		\plotpermbox{10}{12}{13}{15}; 
		\plotpermbox{14}{12}{15}{15}; 
		\node [below right] at ({0+0.25},{6+0.65}) {\scriptsize $1$};
		\node [below right] at ({2+0.25},{11+0.65}) {\scriptsize $3$};
		\node [below right] at ({9+0.25},{15+0.65}) {\scriptsize $5$};
		\node [above left] at ({10-0.25},{0+0.35}) {\scriptsize $2$};
		\node [above left] at ({14-0.25},{6+0.35}) {\scriptsize $4$};
		\node [above left] at ({16-0.25},{11+0.35}) {\scriptsize $6$};
    \end{tikzpicture} 
\quad
    \begin{tikzpicture}[scale={(15*0.25)/3}]
		\draw [thick, dotted, line cap=round] (0.2,0.2)--(1.2,0.2);
		\draw [thick, dotted, line cap=round] (0.3,0.3)--(1.3,0.3);
		\draw [thick, dotted, line cap=round] (1.1,0.1)--(1.1,1.1);
		\draw [thick, dotted, line cap=round] (1.35,0.35)--(1.35,1.35);
		\draw [thick, dotted, line cap=round] (1.7,0.7)--(1.7,1.7);
		\draw [thick, dotted, line cap=round] (1.9,0.9)--(1.9,1.9);
		\draw [thick, dotted, line cap=round] (1.4,1.4)--(2.4,1.4);
		\draw [thick, dotted, line cap=round] (1.5,1.5)--(2.5,1.5);
		\draw [thick, dotted, line cap=round] (1.8,1.8)--(2.8,1.8);
		\draw [thick, dotted, line cap=round] (2.55,1.55)--(2.55,2.55);
		\draw [thick, dotted, line cap=round] (2.75,1.75)--(2.75,2.75);
		\draw [thick, dotted, line cap=round] (2.65,2.65)--(3.65,2.65);
		\draw [thick, dotted, line cap=round] (2.85,2.85)--(3.85,2.85);
		\draw [thick, line cap=round] (0,0)--(3,3);
		\draw [thick, line cap=round] (1,0)--(4,3);
		\draw [darkgray, very thick, rounded corners=0.01, line cap=round] (0,0)--(0,1)--(1,1)--(1,2)--(2,2)--(2,3)--(4,3);
		\draw [darkgray, very thick, rounded corners=0.01, line cap=round] (0,0)--(1,0)--(1,1)--(2,1)--(2,2)--(3,2)--(3,3);
		\draw [darkgray, very thick, rounded corners=0.01, line cap=round] (1,0)--(2,0)--(2,1)--(3,1)--(3,2)--(4,2)--(4,3);
		\plotpartialperm{0.2/0.2, 0.3/0.3, 1.4/1.4, 1.5/1.5, 1.8/1.8, 2.65/2.65, 2.85/2.85};
		\plotpartialperm{1.1/0.1, 1.35/0.35, 1.7/0.7, 1.9/0.9, 2.55/1.55, 2.75/1.75, 3.75/2.75, 3.9/2.9};
		\node [below right] at ({0-0.05},{1+0.04}) {\scriptsize $1$};
		\node [below right] at ({1-0.05},{2+0.04}) {\scriptsize $3$};
		\node [below right] at ({2-0.05},{3+0.04}) {\scriptsize $5$};
		\node [above left] at ({2+0.05},{0-0.04}) {\scriptsize $2$};
		\node [above left] at ({3+0.05},{1-0.04}) {\scriptsize $4$};
		\node [above left] at ({4+0.05},{2-0.04}) {\scriptsize $6$};
    \end{tikzpicture} 
\end{center}
\caption{The greedy staircase gridding of the permutation $2\ 3\ 1\ 4\ 7\ 8\ 5\ 11\ 6\ 9\ 12\ 10\ 14\ 13\ 15$ from Figure~\ref{fig-example-permutations-decomposition} and a drawing of this gridding on two parallel lines. The dotted lines in the picture on the right are included only to indicate relative positions.}
\label{fig-example-permutations-twolines}
\end{figure}

Suppose that we take our figure to consist of the two parallel rays $y=x$ and $y=x-1$ for $y\ge 0$. From any staircase gridding of a $321$-avoiding permutation $\pi$ we can construct a drawing of $\pi$ on these two parallel rays. First we add vertical and horizontal lines $x=i$ and $y=i$ for all natural numbers $i$, splitting the figure into cells. To draw $\pi$ on this figure, take any staircase gridding of $\pi$ and embed it cell by cell into the corresponding cells of the figure, making sure that the relative order between entries in adjacent cells is preserved. An example is shown in Figure~\ref{fig-example-permutations-twolines}.

\section{Domino and Omnibus Encodings}
\label{sec-omnibus}

From any (not necessarily greedy) staircase gridding we construct dominoes. For each $i\ge 0$, the \emph{$i\th$ domino} consists of the entries of the staircase gridding in the $i\th$ and $(i+1)\first$ cells. We then read the entries of this domino in the order specified before (left-to-right for vertically adjacent cells, and bottom-to-top for horizontally adjacent cells), recording the labels of their cells as the \emph{$i\th$ domino factor} $d_i$. Note that both the $0\th$ and final domino factors encode single cells. An example of dominoes and domino factors is shown in Figure~\ref{fig-example-permutations-decomposition-domino}.

\begin{figure}
\[
	\begin{array}{ccccccc}
	    \begin{tikzpicture}[scale=.22, baseline=0]
		    \plotpartialperm{1/2,2/3};
			\plotpermbox{1}{1}{2}{6}; 
			\node [below right] at ({0+0.25},{6+0.65}) {\scriptsize $1$};
			\draw [->, thick] (0.5,-0.1)--(2.5,-0.1);
	    \end{tikzpicture}
	&
	    \begin{tikzpicture}[scale=.22, baseline=0]
		    \plotpartialpermhollow{1/2,2/3};
		    \plotpartialperm{3/1,4/4,7/5,9/6};
			\plotpermbox{1}{1}{2}{6}; 
			\node [below right] at ({0+0.25},{6+0.65}) {\scriptsize $1$};
			\plotpermbox{3}{1}{9}{6}; 
			\node [above left] at ({10-0.25},{0+0.35}) {\scriptsize $2$};
			\draw [->, thick] (-0.1,0.5)--(-0.1,6.5);
	    \end{tikzpicture}
	&
	    \begin{tikzpicture}[scale=.22, baseline=0]
		    \plotpartialpermhollow{3/1,4/4,7/5,9/6};
		    \plotpartialperm{5/7,6/8,8/11};
			\plotpermbox{3}{1}{9}{6}; 
			\node [above left] at ({10-0.25},{0+0.35}) {\scriptsize $2$};
			\plotpermbox{3}{7}{9}{11}; 
			\node [below right] at ({2+0.25},{11+0.65}) {\scriptsize $3$};
			\draw [->, thick] (2.5,-0.1)--(9.5,-0.1);
	    \end{tikzpicture}
	&
	    \begin{tikzpicture}[scale=.22, baseline=0]
		    \plotpartialpermhollow{5/7,6/8,8/11};
		    \plotpartialperm{10/9,12/10};
			\plotpermbox{3}{7}{9}{11}; 
			\node [below right] at ({2+0.25},{11+0.65}) {\scriptsize $3$};
			\plotpermbox{10}{7}{13}{11}; 
			\node [above left] at ({14-0.25},{6+0.35}) {\scriptsize $4$};
			\draw [->, thick] (1.9,6.5)--(1.9,11.5);
	    \end{tikzpicture}
	&
	    \begin{tikzpicture}[scale=.22, baseline=0]
		    \plotpartialpermhollow{10/9,11/12,12/10,13/14};
		    \plotpartialperm{11/12,13/14};
			\plotpermbox{10}{7}{13}{11}; 
			\node [above left] at ({14-0.25},{6+0.35}) {\scriptsize $4$};
			\plotpermbox{10}{12}{13}{15}; 
			\node [below right] at ({9+0.25},{15+0.65}) {\scriptsize $5$};
			\draw [->, thick] (9.5,5.9)--(13.5,5.9);
	    \end{tikzpicture}
	&
	    \begin{tikzpicture}[scale=.22, baseline=0]
		    \plotpartialpermhollow{11/12,13/14,14/13,15/15};
		    \plotpartialperm{14/13,15/15};
			\plotpermbox{10}{12}{13}{15}; 
			\node [below right] at ({9+0.25},{15+0.65}) {\scriptsize $5$};
			\plotpermbox{14}{12}{15}{15}; 
			\node [above left] at ({16-0.25},{11+0.35}) {\scriptsize $6$};
			\draw [->, thick] (8.9,11.5)--(8.9,15.5);
	    \end{tikzpicture}
	&
	    \begin{tikzpicture}[scale=.22, baseline=0]
		    \plotpartialpermhollow{14/13,15/15};
			\plotpermbox{14}{12}{15}{15}; 
			\node [above left] at ({16-0.25},{11+0.35}) {\scriptsize $6$};
			\draw [->, thick] (13.5,10.9)--(15.5,10.9);
	    \end{tikzpicture}
	\\[5pt]
		11
		&
		211222
		&
		2233232
		&
		33443
		&
		4545
		&
		5656
		&
		66
	\end{array}
\]
\caption[The domino factors (bottom row) corresponding to dominoes (top row) of the gridded permutation from Figure~\ref{fig-example-permutations-decomposition}.]{The domino factors (bottom row) corresponding to dominoes (top row) of the gridded permutation from Figure~\ref{fig-example-permutations-decomposition}. The domino encoding of this permutation is therefore $\newpoint\newpoint\# \newpoint\oldpoint\oldpoint\newpoint\newpoint\newpoint\# \oldpoint\oldpoint\newpoint\newpoint\oldpoint\newpoint\oldpoint\# \oldpoint\oldpoint\newpoint\newpoint\oldpoint\# \oldpoint\newpoint\oldpoint\newpoint\# \oldpoint\newpoint\oldpoint\newpoint\# \oldpoint\oldpoint\#$.}
\label{fig-example-permutations-decomposition-domino}
\end{figure}

We now translate the $i\th$ domino factor $d_i$ of the staircase gridded permutation $\pi^\gridded$ to the alphabet $\{\oldpoint,\newpoint\}$ by replacing occurrences of $i$ by $\oldpoint$ and occurrences of $i+1$ by $\newpoint$, labeling the resulting word $d_i^{\newpoint}$. The \emph{domino encoding}, $\delta$, of the gridded permutation $\pi^\gridded$ is then
\[
	\delta(\pi^\gridded) = d_0^{\newpoint} \#d_1^{\newpoint} \#\cdots\# d_m^{\newpoint} \#,
\]
where $m$ is the last nonempty cell. Recall that the relative position of entries in cells $j$ and $k$ is determined by the cells themselves if $|k-j|\ge 2$. Therefore, as the domino factors completely determine the relative positions between entries of adjacent cells, the domino encoding is an injection (as a mapping from staircase gridded permutations to valid domino encodings).

We also derive a second encoding, the \emph{omnibus encoding}, which again collects the domino factors of $\pi^\gridded$ into a single word but this time by interleaving them with each other. In this encoding, for which the alphabet is the positive integers, each entry corresponds to a single letter and the encoding contains every domino factor as a subword. Formally, this means that we insist that the omnibus encoding, $w$, of $\pi^\gridded$ satisfy the \emph{projection condition}:
\begin{enumerate}[label=($\text{\textup{PC}}$), leftmargin=*, widest=1]
	\item $\eval{w}^{\{i,i+1\}}$ is equal to the $i\th$ domino factor of $\pi^\gridded$ for all $i$.
\end{enumerate}
That is, when we look only at the letters $i$ and $i+1$ of an omnibus encoding we recover the $i\th$ domino factor $d_i$. This rule alone does not determine the encoding uniquely because it does not specify the order in which letters belonging to different domino factors should occur. In particular, if $|j-i|\ge 2$ then the letters $i$ and $j$ ``commute'' in the sense that replacing an $ij$ factor by $ji$ does not change the projections to domino factors. We choose to ``prefer'' letters of larger value moving to the left. It is easy to see that this is equivalent to stipulating that our encoding $w$ satisfy the
\emph{small ascent condition}:
\begin{enumerate}[label=\textup{(SAC)}, leftmargin=*, widest=1]
	\item $w(i+1)\le w(i)+1$ for all relevant indices $i$.
\end{enumerate}

The conditions (PC) and (SAC) together guarantee the existence and uniqueness of the omnibus encoding (of gridded permutations). We prove uniqueness momentarily, after demonstrating existence by showing how to compute the omnibus encoding from the domino factors for the example shown in Figure~\ref{fig-example-permutations-decomposition-domino}. Having written one domino factor $d_i$, in the next row we copy the occurrences of $i+1$, and then insert the occurrences of $i+2$ as far to the left as possible, subject to the requirement that the word in that row is $d_{i+1}$.
\[
	\fnarray{rcccccccccccccccc}{%
		d_0&=&&1&1&&&&&&&&&&&&\\
		d_1&=&2&1&1&2&&&&&&&&&2&&2\\
		d_2&=&2&&&2&3&3&&&&&&&2&3&2\\
		d_3&=&&&&&3&3&4&&&4&&&&3&\\
		d_4&=&&&&&&&4&5&&4&5&&&&\\
		d_5&=&&&&&&&&5&6&&5&6&&&\\
		d_6&=&&&&&&&&&6&&&6&&&\\\cline{3-17}\\[-7.5pt]
		&&2&1&1&2&3&3&4&5&6&4&5&6&2&3&2
	}
\]
Clearly this procedure can be performed for the set of domino factors of any gridded $321$-avoiding permutation $\pi^\gridded$, and thus every such gridded permutation has at least one omnibus encoding. To establish uniqueness, we show that every word of positive integers satisfying the small ascent condition is uniquely determined by its projections to pairs of consecutive integers.

\begin{proposition}
\label{prop-consecutive-is-enough}
If the words $u,w\in\mathbb{P}^\ast$ both satisfy the small ascent condition, $\eval{u}^{\{1\}}=\eval{w}^{\{1\}}$, and $\eval{u}^{\{i,i+1\}}=\eval{w}^{\{i,i+1\}}$ for every positive integer $i$, then $u=w$.
\end{proposition}
\begin{proof}
For a positive integer $k$, let $[k] = \{1, 2,\dots, k\}$. We prove inductively that under the hypotheses of the proposition, we have $\eval{u}^{[i]} = \eval{w}^{[i]}$ for all $i\geq 1$. The hypotheses give the base case of $i=1$. Suppose now that $\eval{u}^{[i]} = \eval{w}^{[i]}$ for some $i\ge 1$ and consider any occurrence of $i+1$ in $\eval{u}^{[i+1]}$. If this $i+1$ has any smaller elements to its left, then the rightmost such must equal $i$ owing to the small ascent condition. Therefore $\eval{u}^{[i+1]}$ is formed from $\eval{u}^{[i]}$ by inserting all occurrences of $i+1$ correctly according to $\eval{u}^{\{i,i+1\}}$ and as far to the left as possible subject to this constraint. Since $\eval{w}^{[i+1]}$ is formed from $\eval{w}^{[i]}$ by the same rule and since both $\eval{u}^{[i]} = \eval{w}^{[i]}$ and $\eval{u}^{\{i,i+1\}} = \eval{w}^{\{i,i+1\}}$ it follows that $\eval{u}^{[i+1]} = \eval{w}^{[i+1]}$, completing the proof.
\end{proof}

These facts allow us to define the \emph{omnibus encoding}, $\omega$, from the set of all staircase gridded $321$-avoiding permutation to $\mathbb{P}^\ast$ as the mapping sending $\pi^\gridded$ to the unique word satisfying both the \textup{(PC)} and \textup{(SAC)}. The two languages of most interest to us are then
\begin{eqnarray*}
	\L^{\infty}&=&\{\omega(\pi^\gridded)\st \text{$\pi^\gridded$ is a gridded $321$-avoiding permutation}\} \text{ and}\\
	\G^{\infty}&=&\{\omega(\pi^\gridded)\st \text{$\pi^\gridded$ is a greedily gridded $321$-avoiding permutation}\}.
\end{eqnarray*}
For most of the argument it is easier to ignore the greediness conditions and focus on $\L^\infty$, which has a simple alternative characterisation:
\[
	\L^{\infty}=\{w\in\mathbb{P}^\ast\st \text{$w$ satisfies (SAC)}\}.
\]
Translating the gridding conditions (G1) and (G2) to omnibus encodings, we immediately obtain the following characterisation of the language $\G^{\infty}$.

\begin{observation}
\label{obs-greedy-omegaG}
The word $w\in\mathbb{P}^\ast$ lies in $\G^{\infty}$ if and only if it satisfies the (SAC), i.e., it lies in $\L^{\infty}$, and also the following two conditions:
\begin{enumerate}[label=\textup{($\omega$G\arabic*)}, leftmargin=*, widest=1]
\item For all $i\ge 1$, the first occurrence of $i+1$ occurs before the first occurrence of $i+2$.
\item For all $i\ge 1$, the first occurrence of $i+1$ occurs before the last occurrence of $i$.
\end{enumerate}
\end{observation}

Given any word $w\in\mathbb{P}^\ast$, we define its \emph{$i\th$ domino factor} $d_i$ to be $\eval{w}^{\{i,i+1\}}$, i.e., the subword of $w$ made up of its letters equal to $i$ and $i+1$. In this way, the domino factors of any gridded $321$-avoiding permutation $\pi^\gridded$ are equal to the domino factors of its omnibus encoding $\omega(\pi^\gridded)$. In the same manner, we can define the \emph{domino encoding} of any word $w\in\mathbb{P}^\ast$ as
\[
	\delta(w)=d_0^{\newpoint}\#d_1^{\newpoint}\#\cdots\#d_m^{\newpoint}\#,
\]
where $m$ is the value of the largest letter in $w$.

Therefore given any omnibus encoding $w\in\L^\infty\subseteq\mathbb{P}^\ast$, we can recover the domino factors (or, equivalently, the domino encoding) of the underlying gridded permutation and then, by our previous remarks, reconstruct this gridded permutation. Thus while we initially defined $\L^\infty$ and $\G^\infty$ as images under $\omega$, we now know that they are in fact in one-to-one correspondence with these sets of gridded permutations.

\begin{observation}
\label{obs-omega-bijection}
The mapping $\omega$ is a bijection between the set of gridded (resp., greedily gridded) $321$-avoiding permutations and $\L^\infty$ (resp., $\G^\infty$).
\end{observation}


This observation of course implies that the number of words of length $n$ in $\G^\infty$ is equal to the $n$th Catalan number. The authors asked on MathOverflow~\cite{vatter:a-family-of-wor:} for a simple bijection between (a variant of) this language and another Catalan family (other than staircase griddings). In response, Speyer~\cite{speyer:a-double-gradin:} conjectured a link to the Catalan matroid of Ardila~\cite{ardila:the-catalan-mat:} that was subsequently proved by Stump~\cite{stump:on-a-new-collec:} using Haglund's zeta map~\cite{haglund:the-qt-catalan-:}. Mansour and Shattuck~\cite{Mansour:Chebyshev-polyn:} have since provided several refinements of the enumeration, such as the number of words in the language with a specified number of occurrences of $1$ and $2$.

The domino encoding may appear at first to be superior to the omnibus encoding because the former is defined on the finite alphabet $\{\oldpoint,\newpoint,\#\}$ whereas the latter is defined on the infinite alphabet of positive integers. However, in the context of establishing a regularity result for a subclass $\C\subsetneq\Av(321)$, the domino encoding is of no immediate use. If $\C$ is not finite then it must contain arbitrarily long increasing sequences, and this already implies that the domino encodings of the greedy griddings of members of $\C$ do not form regular language, owing to the condition that the number of $\newpoint$ symbols in the $\{\newpoint, \oldpoint\}$ factor preceding a punctuation mark must equal the number of $\oldpoint$ symbols in the subsequent such factor. Nonetheless, as well as providing a foundation for the omnibus encoding, the domino encoding becomes useful again in the final stages of the proof of Theorem~\ref{thm-321-rational}.

We say that the omnibus encoding is an \emph{entry-to-entry mapping} because every letter of $\omega(\pi)$ corresponds to precisely one entry of $\pi$. The domino encoding is nearly an entry-to-entry mapping because each entry of $\pi$ corresponds to precisely two non-punctuation letters of $\delta(\pi)$. We make frequent, though implicit, use of these correspondences.

The inverse of the omnibus encoding has a natural geometric interpretation, which can be viewed as an infinite version of the encodings defined in~\cite{albert:geometric-grid-:}. Following the notation there we denote the inverse of $\omega$ by $\bij^\gridded$, which is a surjection from $\mathbb{P}^\ast$ to the set of gridded $321$-avoiding permutations, interpreted as equivalence classes of sets of points on the two parallel rays $y=x$ and $y=x-1$ for $y\ge 0$ subdivided into cells by the vertical and horizontal lines $x=i$ and $y=i$ for all integers $i$.

Suppose that the word $w\in\mathbb{P}^\ast$ has length $n$ and choose arbitrary real numbers $0<d_1<\cdots<d_n<1$. For each $1\le i\le n$, take $p_i$ to be the point on the diagonal line segment in the cell numbered by $w(i)$ that is at infinity-norm distance $d_i$ from the lower left corner of this cell. We define $\bij^\gridded(w)$ to be the gridded permutation that is order isomorphic to the gridded set $\{p_1,p_2,\dots,p_n\}$ of points in the plane and we further define $\bij(w)$ to be the permutation obtained from $\bij^\gridded(w)$ by ``forgetting'' the grid lines. It is routine to show that $\bij^\gridded(w)$ does not depend on the particular choice of $d_1, \dots, d_n$, and thus is well-defined. Given any two words $u,w\in\mathbb{P}^\ast$, it is clear from this construction that if $u$ is a subword of $w$ then $\bij(u)\le\bij(w)$. Reframing this observation in terms of the omnibus encoding we obtain the following.

\begin{observation}
\label{obs-order-pres}
Let $\sigma^\gridded$ and $\pi^\gridded$ be gridded $321$-avoiding permutations. If $\omega(\sigma^\gridded)$ is a subword of $\omega(\pi^\gridded)$ then $\sigma\le\pi$.
\end{observation}

\section{Restricting to a Finite Alphabet}

In order to appeal to the theory of formal languages we must translate the omnibus encoding to a finite alphabet. This---accomplished via the panel encoding---is the topic of the next section. Aside from restricting to a finite alphabet though, some other restriction is needed because $\Av(321)$ does not have a rational generating function. This section introduces a generic family of restrictions on the omnibus encodings in such a way that for any proper subclass of $\Av(321)$ one of the restrictions in the family is satisfied. This will subsequently be shown to be sufficient to enable encodings of finitely based and/or well-quasi-ordered subclasses into regular languages over finite alphabets.


Given a word $w\in\mathbb{P}^\ast$ its \emph{shift by $k$} is defined by
\[
	w^{+k}(i) = w(i) + k
\]
for all indices $i$. An \emph{even shift} is a shift by an even integer. By the definition of $\bij$, it follows immediately that $\bij(w^{+2k})=\bij(w)$, so the image of $\bij$ is unaffected by even shifts. As a consequence of this fact and Observation~\ref{obs-order-pres}, we obtain the following.

\begin{observation}
\label{obs-contain-omnibus-matchings}
Let $\sigma^\gridded$ and $\pi^\gridded$ be gridded $321$-avoiding permutations. If any even shift of $\omega(\sigma^\gridded)$ is a subword of $\omega(\pi^\gridded)$ then $\sigma\le\pi$.
\end{observation}

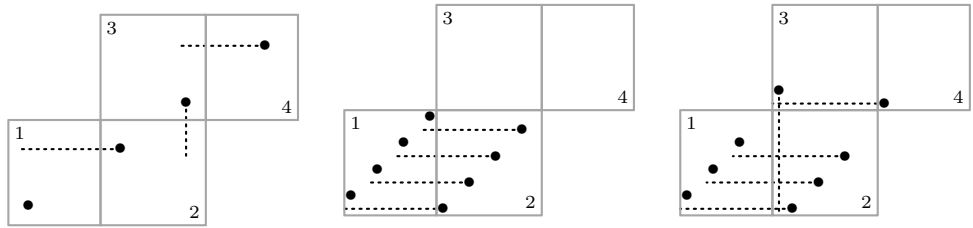
\begin{figure}
\begin{center}
    \begin{tikzpicture}[scale=0.175]
		\draw [thick, dotted, line cap=round] (6,6.28571428571)--(-1.5,6.28571428571);
		\draw [thick, dotted, line cap=round] (11,9.85714285714)--(11,5.78571428571);
		\draw [thick, dotted, line cap=round] (17,14.1428571429)--(10.5,14.1428571429);
		\plotpartialperm{-1/2,6/6.28571428571,11/9.85714285714,17/14.1428571429};
		\plotpermbox{-2}{1}{4}{8}; 
		\plotpermbox{5}{1}{12}{8}; 
		\plotpermbox{5}{9}{12}{16}; 
		\plotpermbox{13}{9}{19}{16}; 
		\node [below right] at ({-3+0.25},{8+0.65}) {\scriptsize $1$};
		\node [below right] at ({4+0.25},{16+0.65}) {\scriptsize $3$};
		\node [above left] at ({13-0.25},{0+0.35}) {\scriptsize $2$};
		\node [above left] at ({20-0.25},{8+0.35}) {\scriptsize $4$};
    \end{tikzpicture} 
\quad
    \begin{tikzpicture}[scale=0.175]
		\draw [thick, dotted, line cap=round] (5,1)--(-2.5,1);
		\draw [thick, dotted, line cap=round] (7,3)--(-0.5,3);
		\draw [thick, dotted, line cap=round] (9,5)--(1.5,5);
		\draw [thick, dotted, line cap=round] (11,7)--(3.5,7);
		\plotpartialperm{-2/2,0/4,2/6,4/8,5/1,7/3,9/5,11/7};
		\plotpermbox{-2}{1}{4}{8}; 
		\plotpermbox{5}{1}{12}{8}; 
		\plotpermbox{5}{9}{12}{16}; 
		\plotpermbox{13}{9}{19}{16}; 
		\node [below right] at ({-3+0.25},{8+0.65}) {\scriptsize $1$};
		\node [below right] at ({4+0.25},{16+0.65}) {\scriptsize $3$};
		\node [above left] at ({13-0.25},{0+0.35}) {\scriptsize $2$};
		\node [above left] at ({20-0.25},{8+0.35}) {\scriptsize $4$};
    \end{tikzpicture} 
\quad
    \begin{tikzpicture}[scale=0.175]
		\draw [thick, dotted, line cap=round] (6,1)--(-2.5,1);
		\draw [thick, dotted, line cap=round] (8,3)--(-0.5,3);
		\draw [thick, dotted, line cap=round] (10,5)--(1.5,5);
		\draw [thick, dotted, line cap=round] (5,10)--(5,0.5);
		\draw [thick, dotted, line cap=round] (13,9)--(4.5,9);
		\plotpartialperm{-2/2,0/4,2/6,5/10,6/1,8/3,10/5,13/9};
		\plotpermbox{-2}{1}{4}{8}; 
		\plotpermbox{5}{1}{12}{8}; 
		\plotpermbox{5}{9}{12}{16}; 
		\plotpermbox{13}{9}{19}{16}; 
		\node [below right] at ({-3+0.25},{8+0.65}) {\scriptsize $1$};
		\node [below right] at ({4+0.25},{16+0.65}) {\scriptsize $3$};
		\node [above left] at ({13-0.25},{0+0.35}) {\scriptsize $2$};
		\node [above left] at ({20-0.25},{8+0.35}) {\scriptsize $4$};
    \end{tikzpicture} 
\end{center}
\caption{The picture on the left shows the gridding of the permutation $1234$ encoded by the word $1234$. The pictures in the centre and on the right shown two different griddings of the permutation $24681357$ encoded, respectively, by the words $21212121$ and $43212121$.}
\label{fig-example-1234-obstruction-1}
\end{figure}

The converse of this observation does not hold---a simple example is given by the permutations $\sigma=1234$ and $\pi=24681357$. Taking $\sigma^\gridded$ to be the gridding of $\sigma$ encoded by the word $1234$ shown on the left of Figure~\ref{fig-example-1234-obstruction-1} and $\pi^\gridded$ to be the gridding of $\pi$ encoded by the word $21212121$ shown in the centre of Figure~\ref{fig-example-1234-obstruction-1}, we see that although $\sigma$ is contained in $\pi$, $\omega(\pi^\gridded)$ contains no shift, let alone an even one, of $\omega(\sigma^\gridded)$. Indeed, it can be argued that \emph{no} encoding of $\pi$ contains a shift of the word $1234$---while there are griddings of $\pi$ that use four cells, such as the one shown on the right of Figure~\ref{fig-example-1234-obstruction-1}, none of the encodings of those griddings can contain an increasing subsequence of length four. Thus even a weak converse to Observation~\ref{obs-contain-omnibus-matchings} cannot hold, and indeed the question ``given omnibus encodings of $\sigma$ and $\pi$, determine whether $\sigma\le\pi$'' is difficult to answer with the tools we have at hand. Fortunately, this is not a question we will have to address directly; when it comes time for us to enforce avoidance conditions in Section~\ref{sec-detecting-basis-elts}, we transduce from omnibus encodings to Dyck paths and enforce avoidance conditions in that setting.

In the next two sections we focus on the languages
\[
	\L_c^\infty
	=
	\{w\in\mathbb{P}^\ast\st\text{$w$ satisfies the small ascent condition and avoids all shifts of $(12\cdots c)^c$}\}.
\]
This definition is justified by the following proposition. (Note that, as the proof shows, the condition on avoiding all shifts of $(12\cdots c)^c$ could be weakened, but we have no need to do so.)

\begin{proposition}
\label{prop-avoid-shifts}
For every proper subclass $\C$ of $321$-avoiding permutations there is a positive integer $c$ such that $\omega(\pi^\gridded)\in\L_c^\infty$ for all staircase griddings $\pi^\gridded$ of permutations $\pi\in\C$.
\end{proposition}
\begin{proof}
Let $\beta$ be any $321$-avoiding permutation not belonging to $\C$ with greedy gridding $\beta^\gridded$ and set $c = |\beta| + 1$. Clearly $\omega(\beta^\gridded)$ is contained in $(1 2 \cdots (c-1))^{c-1}$ and so no word of the form $\omega(\pi^\gridded)$ for $\pi\in\C$ may contain an even shift of $(1 2 \cdots (c-1))^{c-1}$ by Observation~\ref{obs-contain-omnibus-matchings}. Moreover, the word $12\cdots c$ contains both $(1 2 \cdots (c-1) )^{+0}$ and $(1 2 \cdots (c-1) )^{+1}$, so any shift of $(1 2 \cdots c)^{c-1}$ contains an even shift of $(1 2 \cdots (c-1) )^{c-1}$. Therefore no word of the form $\omega(\pi^\gridded)$ for $\pi\in\C$ may contain a shift of $(1 2 \cdots c)^{c-1}$, proving the proposition.
\end{proof}

\begin{figure}
\begin{center}
    \begin{tikzpicture}[scale=0.18]
		\draw [thick, dotted, line cap=round] (5,2)--(-1,2);
		\draw [thick, dotted, line cap=round] (7,4)--(1,4);
		\draw [thick, dotted, line cap=round] (9,6)--(3,6);
		\draw [thick, dotted, line cap=round] (11,8)--(5,8);
		\draw [thick, dotted, line cap=round] (6,9)--(6,3);
		\draw [thick, dotted, line cap=round] (8,11)--(8,5);
		\draw [thick, dotted, line cap=round] (10,13)--(10,7);
		\draw [thick, dotted, line cap=round] (12,15)--(12,9);
		\draw [thick, dotted, line cap=round] (13,10)--(7,10);
		\draw [thick, dotted, line cap=round] (15,12)--(9,12);
		\draw [thick, dotted, line cap=round] (17,14)--(11,14);
		\draw [thick, dotted, line cap=round] (19,16)--(13,16);
		\plotpartialperm{-2/1,0/3,2/5,4/7,5/2,6/9,7/4,8/11,9/6,10/13,11/8,12/15,13/10,15/12,17/14,19/16};
		\plotpermbox{-2}{1}{4}{8}; 
		\plotpermbox{5}{1}{12}{8}; 
		\plotpermbox{5}{9}{12}{16}; 
		\plotpermbox{13}{9}{19}{16}; 
		\node [below right] at ({-3+0.25},{8+0.65}) {\scriptsize $1$};
		\node [below right] at ({4+0.25},{16+0.65}) {\scriptsize $3$};
		\node [above left] at ({13-0.25},{0+0.35}) {\scriptsize $2$};
		\node [above left] at ({20-0.25},{8+0.35}) {\scriptsize $4$};
    \end{tikzpicture} 
\end{center}
\caption{The plot of the gridded permutation $\pi^\gridded$ for which $\omega(\pi^\gridded)=(1234)^4$, or from the geometric viewpoint, the gridded permutation $\bij^\gridded((1234)^4)$.}
\label{fig-example-1234-obstruction-2}
\end{figure}
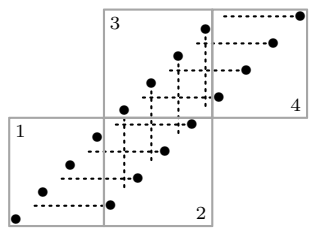


Stated from the geometric perspective, the proof of Proposition~\ref{prop-avoid-shifts} shows that every $321$-avoiding permutation (and consequently, every potential additional basis element of a proper subclass of $\Av(321)$) is contained in $\bij((12\cdots c)^c)$ for some value of $c$. Thus the permutations $\bij((12\cdots c)^c)$ are \emph{universal}%
\footnote{Universal objects for permutation classes are sometimes called \emph{super-patterns}. The typical problem is, given a class $\C$, determine the length of a shortest permutation containing all permutations in $\C_n$. Our universal object is not the shortest possible for the class $\Av(321)$. In particular, Miller~\cite{miller:asymptotic-boun:} found a universal permutation for the class of \emph{all} permutations which has length $(n^2+n)/2$, and Engen and Vatter~\cite{engen:containing-all-:} have recently lowered this bound to $\lceil (n^2+1)/2\rceil$. For the class $\Av(321)$, Bannister, Devanny, and Eppstein~\cite{bannister:small-superpatt:} have constructed a universal permutation of length $22n^{3/2}+\Theta(n)$, while in the case of $\Av(231)$, Bannister, Cheng, Devanny, and Eppstein~\cite{bannister:superpatterns-a:} have established an upper bound of $n^2/4+\Theta(n)$.}
for $\Av(321)$. An example of one of these universal permutations is shown in Figure~\ref{fig-example-1234-obstruction-2}.

%
%
%

%
%
%
%
%

\section{The Panel Encoding \texorpdfstring{$\eta_c$}{}}
\label{sec-panel-encoding}

This section and the next focus solely on the language $\L_c^\infty$ and an encoding, $\eta_c$, which maps it to a language $\L_c^\eta$ over a finite alphabet. The encoding $\eta_c$ is described in this section while the regularity of $\L_c^\eta$ is established in the next. Throughout, consider $c$ to be a fixed positive integer and a word $w \in \L^{\infty}_c$ to be given. Further suppose that the maximum value of a letter in $w$ is $m$.

The material in the remainder of this section is rather technical, so we begin with an overview of the general strategy. Consider the maximal factors of $w$ not containing occurrences of the symbol $1$. Ignoring the (possibly empty) factor which precedes the first $1$ of $w$, all of these factors are immediately preceded by a $1$, and we partition them into two sets. We call one of these factors \emph{large} if it contains an occurrence of the symbol $c$ and \emph{small} if it does not. (The factor which precedes the first $1$ of $w$ is considered neither small nor large, and is handled separately throughout.) By the small ascent condition and the fact that these factors are all preceded by a $1$, each large factor contains an occurrence of $23\cdots c$ which, together with the adjacent $1$, yields an occurrence of $12\cdots c$. Since $w$ avoids $(12\cdots c)^c$ there must be fewer than $c$ large factors. The small ascent condition also implies that each of the small factors forms a word over the alphabet $\{1,\dots,c-1\}$.

The encoding $\eta_c$ begins by performing this separation into small and large factors of $w$. The small factors form a word over $\{1,\dots,c-1\}$ and this word is recorded essentially as is; the large factors are then processed recursively. In order to facilitate the reconstruction of $w$ from its encoding, we need to record the places where the separation occurred. We achieve this by decorating $1$s that are supposed to be followed by the (now removed) large factors, and the matching $2$s at the start of these large factors. Since all letters of the large factors are greater than $1$, we can reduce all of them by $1$ and repeat the process. At each stage the maximum value remaining decreases by at least $1$ (in fact exactly $1$ if there is a large factor present), and so we eventually produce a sequence of (decorated) words over the alphabet $\{1,2,\dots,c-1\}$. The encoding $\eta_c(w)$ is simply the concatenation of these words, separated by punctuation symbols.

Moving to the technical details, we aim to describe an injection $\eta_c : \L_{c}^{\infty} \to \Sigma^*$ where
\[
\Sigma = \{1, 2, \dots, c-1 \} \cup \{ \etaleft{1}, \etaright{1}, \etaleftright{1} \}
\cup \{ \# \}.
\]
We refer to the three symbols $\etaleft{1}$, $\etaright{1}$, $\etaleftright{1}$ as \emph{decorated letters}, and in describing the construction also make use of one more decorated letter: $\etaleft{2}$. Specifically, we have
\[
\eta_c(w) = p_0 \# p_1 \# p_2 \# \cdots \# p_{m-1} \#
\]
(recall that $m$ is the maximum value of a letter in $w$) where none of the factors $p_i$ contains the punctuation symbol $\#$. The words $p_i$ are referred to as \emph{panel words}. Each panel word corresponds to a subword of $w$. More specifically, $p_k^{+k}$ is, after removal of the decorations from any letters, actually a subword of $w$, and together these subwords partition the letters of $w$. Therefore, ignoring the punctuation symbols, $\eta_c$ is an entry-to-entry mapping. The careful reader may note that all panel words of index greater than $m-c+1$ are empty by construction; these are recorded (with punctuation) merely for the sake of consistency.
 

The construction is recursive: we extract the panel words from $w$ in order, starting with $p_0$, so it is convenient to consider also a sequence of \emph{remainder words} $r_0$, $r_1$, \dots, $r_{m-1}$ that represent the as-yet-unencoded letters of $w$.  Each word $r_i$ is defined over the alphabet $\mathbb{P} \cup \{ \etaleft{1} \}$.


The first step of the process is to set $r_0=w$. Suppose that $r_0$ has $k$ letters of value $1$ and express it as
\[
	r_0=r_{0,0}\ 1\ r_{0,1}\ 1\ r_{0,2}\ \cdots\ r_{0,k-1}\ 1\ r_{0,k},
\]
so $r_{0,j}\in (\mathbb{P} \setminus \{1\})^\ast$ for all $j$. Let $J$ denote the set of indices $j$ between $1$ and $k$ such that  $r_{0,j}$ contains a letter of value $c$ (the large factors). Note that $r_{0,0}$ may also contain a letter of value $c$. These factors are precisely what we do \emph{not} encode in the panel word $p_0$. Note that the small ascent condition implies that each non-empty word $r_{0,j}$ for $j \geq 1$ (and particularly all of those with $j\in J$) begins with a $2$. 


We now \emph{decorate} $2|J|$ letters of $r_0$, producing an auxiliary word, $t_0$. We define $t_0$ via factors $t_{0,j}$ for $0\le j\le k$. We first set $t_{0,j}$ equal to $r_{0,j}$ for all $j\notin J$. For $j\in J$, we know that $r_{0,j}$ is nonempty and so begins with a $2$, and we set the corresponding $t_{0,j}$ equal to $r_{0,j}$ with its leftmost $2$ adorned by a $\leftarrow$, turning it into a $\etaleft{2}$, which we call a \emph{left letter}. Each $2$ which is turned into a $\etaleft{2}$ in this process is immediately preceded by a $1$ in $r_0$, which we call a \emph{right letter} and denote by $\etaright{1}$ in $t_0$. After performing these decorations, the auxiliary word $t_0$ can be written as
\[
	t_0=t_{0,0}\ \ell_1 \ t_{0,1}\ \ell_2 \ t_{0,2}\ \cdots\ t_{0,k-1}\ \ell_k \ t_{0,k}
\]
where
\[
	\ell_j=\left\{\begin{array}{ll}
	\etaleft{1}&\text{if $j\in J$ and}\\
	1&\text{if $j\notin J$,}
	\end{array}\right.
\]
and $t_{0,j}$ begins with $\etaleft{2}$ if and only if $j\in J$.

We can now construct our first \emph{panel word}. It is simply the concatenation of all factors $t_{0,j}$ with $j\notin J\cup\{0\}$ (the small factors) and all $\ell_i$, retaining their order in $t_0$. To be precise, if we define
\[
	p_{0,j}=\left\{\begin{array}{ll}
	\epsilon&\text{if $j\in J$ and}\\
	t_{0,j}&\text{if $j\notin J$,}
	\end{array}\right.
\]
then
\[
	p_0=\ell_1\ p_{0,1}\ \ell_2\ p_{0,2}\ \cdots\ p_{0,k-1}\ \ell_k\ p_{0,k}.
\]
The word $p_0$ is thus defined over the alphabet $\{1,2,\dots,c-1\}\cup\{\etaright{1}\}$. Finally, we define $r_1$ to be the result of concatenating the remaining factors $t_{0,j}$ for $j\in J\cup\{0\}$ and then subtracting $1$ from each letter (when we subtract $1$ from a $\etaleft{2}$ we change it to a $\etaleft{1}$). Thus $r_1$ is defined over the alphabet $\mathbb{P} \cup\{\etaleft{1}\}$, and the maximum value of a symbol occurring in $r_1$ is $m-1$.

The decorations of the panel word and the remainder word specify the way to reassemble $w$ from $(p_0,r_1)$. To do so we first form $r_1^{+1}$. We divide this into factors each beginning with a $\etaleft{2}$ (except possibly the first factor). These are the factors $r_{0,j}$ for $j\in J\cup\{0\}$. If $r_1^{+1}$ contains an initial factor that does not begin with $\etaleft{2}$ then we place this factor before $p_0$. Then proceeding from left to right we insert the first of the factors of $r_1^{+1}$ that begin with $\etaleft{2}$ immediately after the first $\etaright{1}$ of $p_1$, then the second factor immediately after the second $\etaright{1}$ of $p_1$, and so on. Effectively, we ``zip together'' $p_0$ and $r_1$ using the arrows to mark the points where the two pieces should mesh with one another. We finish by removing the decorations.

There is only one change in subsequent iterations of this encoding. In constructing $t_j$, $p_j$ and $r_{j+1}$ from $r_j$ for $j\ge 1$, we may wish to designate a $\etaleft{1}$ (a former left letter) as a right letter. If this situation arises, we simply turn the $\etaleft{1}$ into a $\etaleftright{1}$. Thus after we have decorated $r_j$ to form $t_j$, every decorated letter is either a $\etaleft{1}$ or occurs in a $\etaright{1}\: \etaleft{2}$ or $\etaleftright{1} \: \etaleft{2}$ factor. We call the resulting mapping $\psi\st r_j\mapsto (p_j,r_{j+1})$ the \emph{splitting mapping}.

It follows that, as claimed, every panel word is defined over the alphabet $\{1, 2, \cdots, c-1\}\cup\{\etaright{1}, \etaleft{1}, \etaleftright{1}\}$. Moreover, because the greatest letter in $w$ has the value $m$, when we come to construct $p_{m-1}$ from $r_{m-1}$ all (if any) remaining letters are $1$ (or a decorated version thereof) and thus after this stage we can guarantee that we have encoded all of $w$. As promised, our ultimate encoding consists of the concatenation of these panel words, separated by punctuation, $\eta_c(w)=p_0\# p_1\# \cdots \# p_{m-1} \#$.

\newcommand{\ts}{\!\;}	

We illustrate this process with a concrete example. Suppose that $c=4$ and consider encoding the word
\[
	\begin{array}{rcl}
	\phantom{t_0}&&\phantom{\underline{2\ts 3}\ts 1\ts 2\ts 3\ts \etaright{1}\ts \underline{\etaleft{2}\ts 2\ts 3\ts 2\ts 3\ts 4\ts 5\ts 2\ts 3}\ts \etaright{1}\ts \underline{\etaleft{2}\ts 3\ts 4\ts 5}\ts 1\ts 2\ts 1\ts 2\ts 3\ts 2\ts \etaright{1}\ts \underline{\etaleft{2}\ts 3\ts 4\ts 3}.}\\[-12pt]
	w&=&2\ts 3\ts 1\ts 2\ts 3\ts 1\ts 2\ts 2\ts 3\ts 2\ts 3\ts 4\ts 5\ts 2\ts 3\ts 1\ts 2\ts 3\ts 4\ts 5\ts 1\ts 2\ts 1\ts 2\ts 3\ts 2\ts 1\ts 2\ts 3\ts 4\ts 3
	\end{array}
\]
from $\L_{4}^\infty$. In our first step, we set $r_0=w$, divide it into factors, and add decorations (here and in what follows we underline the factors that remain in $r_{j+1}$). We call the resulting decorated word $t_0$. In our example, this step yields
\[
	\begin{array}{rcl}
	\phantom{w}&&\phantom{2\ts 3\ts 1\ts 2\ts 3\ts 1\ts 2\ts 2\ts 3\ts 2\ts 3\ts 4\ts 5\ts 2\ts 3\ts 1\ts 2\ts 3\ts 4\ts 5\ts 1\ts 2\ts 1\ts 2\ts 3\ts 2\ts 1\ts 2\ts 3\ts 4\ts 3}\\[-12pt]
	t_0&=&\underline{2\ts 3}\ts 1\ts 2\ts 3\ts \etaright{1}\ts \underline{\etaleft{2}\ts 2\ts 3\ts 2\ts 3\ts 4\ts 5\ts 2\ts 3}\ts \etaright{1}\ts \underline{\etaleft{2}\ts 3\ts 4\ts 5}\ts 1\ts 2\ts 1\ts 2\ts 3\ts 2\ts \etaright{1}\ts \underline{\etaleft{2}\ts 3\ts 4\ts 3}.
	\end{array}
\]
We then form both $p_0$ and $r_1$ and repeat the process to form $t_1$, computing
\[
	\begin{array}{rclcrcl}
	\phantom{p_1}&\phantom{=}&\phantom{1\ts 2\ts \etaleft{1}\ts 1\ts 2\ts \etaright{1}\ts 1\ts 2\ts \etaleftright{1}\ts \etaleft{1}\ts 2\ts 3\ts 2,}
	&\quad&
	\phantom{r_2}&\phantom{=}&\phantom{\etaleft{1}\ts 2\ts 3\ts \etaleft{1}\ts 2\ts 3,}
	\\[-12pt]
	p_0&=&1\ts 2\ts 3\ts \etaright{1}\ts \etaright{1}\ts 1\ts 2\ts 1\ts 2\ts 3\ts 2\ts \etaright{1},
	&\quad&
	r_1&=&1\ts 2\ts \etaleft{1}\ts 1\ts 2\ts 1\ts 2\ts 3\ts 4\ts 1\ts 2\ts \etaleft{1}\ts 2\ts 3\ts 4\ts \etaleft{1}\ts 2\ts 3\ts 2,\\
	&&&&
	t_1&=&1\ts 2\ts \etaleft{1}\ts 1\ts 2\ts \etaright{1}\ts \underline{\etaleft{2}\ts 3\ts 4}\ts 1\ts 2\ts \etaleftright{1}\ts \underline{\etaleft{2}\ts 3\ts 4}\ts \etaleft{1}\ts 2\ts 3\ts 2.
	\end{array}
\]
The list of panel words and remainders is completed by performing these operations several more times, in which we find
\[
	\begin{array}{rclcrcl}
	\phantom{p_0}&\phantom{=}&\phantom{1\ts 2\ts 3\ts \etaright{1}\ts \etaright{1}\ts 1\ts 2\ts 1\ts 2\ts 3\ts 2\ts \etaright{1},}
	&&
	\phantom{t_1}&\phantom{=}&\phantom{1\ts 2\ts \etaleft{1}\ts 1\ts 2\ts \etaright{1}\ts \underline{\etaleft{2}\ts 3\ts 4}\ts 1\ts 2\ts \etaleftright{1}\ts \underline{\etaleft{2}\ts 3\ts 4}\ts \etaleft{1}\ts 2\ts 3\ts 2.}
	\\[-12pt]
	p_1&=&1\ts 2\ts \etaleft{1}\ts 1\ts 2\ts \etaright{1}\ts 1\ts 2\ts \etaleftright{1}\ts \etaleft{1}\ts 2\ts 3\ts 2,
	&\quad&
	r_2&=&\etaleft{1}\ts 2\ts 3\ts \etaleft{1}\ts 2\ts 3,
	\\
	&&&&
	t_2&=&\etaleft{1}\ts 2\ts 3\ts \etaleft{1}\ts 2\ts 3,
	\\
	p_2&=&\etaleft{1}\ts 2\ts 3\ts \etaleft{1}\ts 2\ts 3,
	&\quad&
	r_3&=&\epsilon,
	\\
	&&&&
	t_3&=&\epsilon,
\end{array}
\]
and all of $p_3$, $r_4$, $t_4$, and $p_4$ are empty. Our encoding of $w$ is the concatenation of the panel words $p_0$, $p_1$, $p_2$, $p_3$, and $p_4$ separated by punctuation,
\[
	\eta_{4}(w)
	=
	1\ts 2\ts 3\ts \etaright{1}\ts \etaright{1}\ts 1\ts 2\ts 1\ts 2\ts 3\ts 2\ts \etaright{1}\ts \#\ts 
	1\ts 2\ts \etaleft{1}\ts 1\ts 2\ts \etaright{1}\ts 1\ts 2\ts \etaleftright{1}\ts \etaleft{1}\ts 2\ts 3\ts 2\ts \#\ts 
	\etaleft{1}\ts 2\ts 3\ts \etaleft{1}\ts 2\ts 3\ts \#\ts \#\ts \#.
\]

Returning to the general situation, we define
\[
	\L^{\eta}_c = \{\eta_c(w)\st w\in\L_c^\infty\}
\]
to be the image of $\L_c^{\infty}$ under $\eta_c$. Except for punctuating letters, the panel encoding $\eta_c$ is an entry-to-entry mapping. We collect this fact, along with further bookkeeping results which follow immediately from the definition of $\eta_c$, below.

\begin{observation}
\label{obs-where-encoded}
The mapping $\eta_c\st\L_c^\infty\to\L_c^{\eta}$ is a bijection. Moreover, under this mapping, the non-punctuation letters of $\eta_c(w)$ are in one-to-one correspondence with the letters of $w$. In addition, an entry of value $j$ in the panel word $p_i$ of $\eta_c(w)$ corresponds to an entry of value $i+j$ in $w$. Hence, every entry of the panel word $p_i$ corresponds to a letter of value between $i+1$ and $i+c-1$ in $w$, while every letter $i$ in $w$ corresponds to an entry in one of the panel words $p_{i-c+1}$, $p_{i-c+2}$, \dots, or $p_{i-1}$ in $\eta_c(w)$.
\end{observation}

It should be clear that $\eta_c$ is injective, but as we shall need some properties of its inverse in what follows, we shall be a little more explicit. We begin by defining $\psi^{-1}$, the inverse of the splitting mapping. This is the mapping that ``zips together'' a panel word and a remainder word as described previously for the case of $p_0$ and $r_1$.

Suppose that $p$ (to be thought of as the most recent panel word) is a word over $\mathbb{P}\cup\{\etaright{1},\etaleft{1},\etaleftright{1}\}$ and $r$ (to be thought of as the most recent remainder word) is a word over $\mathbb{P}\cup\{\etaleft{1}\}$, and that the number of right letters in $p$ equals the number of left letters in $r$ (both are equal to $k$ below). Now express $p$ and $r$ in the form 
\[
	\arraycolsep=4pt
	\begin{array}{rcllllllllcllllll}
	\phantom{\psi^{-1}(p,r)}&\phantom{=}&\phantom{s_0^{+1}}&\phantom{q_0}&\phantom{\dot{\ell}_1}&\phantom{2}&\phantom{s_1^{+1}}&\phantom{q_1}&\phantom{\dot{\ell}_2}&\phantom{2}&\phantom{\cdots}&\phantom{s_{k-1}^{+1}}&\phantom{q_{k-1}}&\phantom{\dot{\ell}_k}&\phantom{2}&\phantom{s_k^{+1}}&\phantom{q_k,}\\[-12pt]
	p&=&&q_0&\ell_1&&&q_1&\ell_2&&\cdots&&q_{k-1}&\ell_k&&&q_k,\\
	r&=&s_0&&&\etaleft{1}&s_1&&&\etaleft{1}&\cdots&s_{k-1}&&&\etaleft{1}&s_k,&
	\end{array}
\]
where each $\ell_j$ is a right letter (i.e., $\etaright{1}$ or $\etaleftright{1}$). The inverse of the splitting mapping is defined by
\[
	\arraycolsep=4pt
	\begin{array}{rcllllllllcllllll}
	\psi^{-1}(p,r)&=&s_0^{+1}&q_0&\dot{\ell}_1&2&s_1^{+1}&q_1&\dot{\ell}_2&2&
	\cdots&
	s_{k-1}^{+1}&q_{k-1}&\dot{\ell}_k&2&s_k^{+1}&q_k,
	\end{array}
\]
where $s_i^{+1}$ is the shift-by-$1$ mapping applied to $s_i$ and
\[
	\dot{\ell}_i=\left\{
	\begin{array}{ll}
	1&\text{if $\ell_i=\etaright{1}$},\\
	\etaleft{1}&\text{if $\ell_i=\etaleftright{1}$}
	\end{array}
	\right.
\]
is the mapping that removes right arrows.


Supposing that $p_0$, $p_1$, \dots, $p_{m-1}$ are words over $\mathbb{P}\cup\{\etaright{1},\etaleft{1},\etaleftright{1}\}$ and that the number of right letters of $p_i$ is equal to the number of left letters of $p_{i+1}$ for $0 \leq i < m-1$ we can define
\[
	\Psi^{-1}(p_0\#p_1\#\cdots\#p_{m-1}\#)
	=
	\psi^{-1}(p_0,
	\psi^{-1}(\dots
	\psi^{-1}(p_{m-3}, \psi^{-1}(p_{m-2},p_{m-1}))
	\dots)).
\] 

For $w \in \L_c^{\infty}$, $\Psi^{-1} ( \eta_c(w)) = w$, so $\eta_c$ is indeed an injection on $\L_c^{\infty}$ (and in this context we often write $\eta_c^{-1}$ in place of $\Psi^{-1}$).

There are several features of $\psi^{-1}$ that are important to draw attention to. First, except for removing some decoration and incrementing $r$ by one, $\psi^{-1}$ does not change the subwords $p$ and $r$ at all; that is, absent decoration, $p$ and $r^{+1}$ occur as subwords in  $\psi^{-1}(p,r)$. Second, undecorated letters play no significant role in the reassembly process performed by $\psi^{-1}$, in fact their only role is to be copied into the output (possibly after incrementation). Thus if we delete an undecorated letter from $\eta_c(w)$ and then apply $\Psi^{-1}$ the result is $w$ with the corresponding letter deleted.

The next result provides the interface that we need later to impose basis conditions on panel encodings.

\begin{proposition}
\label{prop-delete-panels}
Let $u$ be a subword of $\eta_c(w)$ whose letters occur in the contiguous set of panel words $p_i$, $p_{i+1}$, $\dots$, $p_{i+k}$. The relative positions of the letters of $w$ corresponding to those in $u$ are determined by the subword of $\eta_c(w)$ consisting of the letters in $u$ together with all decorated letters of the panel words $p_i$, $p_{i+1}$, $\dots$, $p_{i+k}$ and the punctuation symbols $\#$ between them.
\end{proposition}
\begin{proof}
Write $\eta_c(w)=p_0\#p_1\#\cdots\#p_{m-1}\#$ and consider the process of inverting the $\eta_c$ mapping. Once we have formed a word containing all of the letters of $u$ we may stop,
so it suffices to compute
\[
	\psi^{-1}(p_i, \dots \psi^{-1}(p_{i+k},
		\psi^{-1}(p_{i+k+1}, \dots \psi^{-1}(p_{m-2},p_{m-1}) \dots
	)) \dots )
	=
	\psi^{-1}(p_i, \dots \psi^{-1}(p_{i+k},
		r
	) \dots ),
\]
where $r=\psi^{-1}(p_{i+k+1}, \dots \psi^{-1}(p_{m-2},p_{m-1}) \dots)$. In $\psi^{-1}(p_{i+k},r)$, the letters corresponding to $r$ have lost their decoration, and thus may be forgotten by our observation above. Thus it suffices to compute
\[
	\psi^{-1}(p_i, \dots \psi^{-1}(p_{i+k-1}, p_{i+k}) \dots ).
\]
Applying our observation again, we may remove all undecorated letters not belonging to $u$ from these panels without affecting the eventual order of the letters corresponding to $u$. What remains is the information specified in the statement of the proposition (the punctuation symbols serving to distinguish $p_i$ through $p_{i+k}$).
\end{proof}

\section{The Regularity of \texorpdfstring{$\L_c^\eta$}{the Image}}

Our ultimate aim is to establish that various sublanguages of $\L_c^\eta$ (corresponding to finitely based or well-quasi-ordered subclasses of $321$-avoiding permutations) are regular. We first establish that $\L_c^\eta$ itself is regular. The material in this section is also somewhat technical so we again provide an initial informal discussion. We seek to recognise whether a word over the alphabet $\{1, 2, \dots, c-1 \} \cup \{ \etaleft{1}, \etaright{1}, \etaleftright{1} \} \cup \{ \# \}$ belongs to $\L_{c}^{\eta}$. The basic idea is to identify several necessary conditions which are collectively sufficient. Then, if we verify that each individual necessary condition corresponds to a regular language, the closure of regular languages under the Boolean operations proves the result we want. Roughly speaking there are three such necessary conditions: a translation of the small ascent condition, consistency in left-right decorations between consecutive panel words (here the fact that the number of such decorations is bounded is critical, a fact that follows from our observation in the previous section that the number of large factors is bounded), and that the number of panel words captures the maximum letter(s) properly, i.e., that the encoding is not terminated too early or too late. 

We begin with a more detailed look at various properties of panel words, remainder words, and the encodings $\eta_c(w)$. Denote by $\leftcount(p)$ and $\rightcount(p)$ the number of left letters and right letters of a word $p$ (occurrences of $\etaleftright{1}$ contribute to both counts).

Consider the language of all remainder words $r=r_j$ which could arise in the process of encoding words from $\L_c^\infty$. This is a language over the alphabet $\mathbb{P}\cup\{\etaleft{1}\}$. For $j=0$, $r$ is an arbitrary element of $\L_c^\infty$; in particular $\L_c^\infty$ is contained within the language under consideration. Otherwise, $r$ is obtained from an earlier remainder word by marking the left and right letters, concatenating the factor preceding the initial $1$ with the factors from any $\etaleft{2}$ up to but not including the subsequent (marked or unmarked) $1$ and then reducing the value of all letters by $1$. It follows inductively that any such word $r$ satisfies the following three conditions.
\begin{enumerate}[label=\textup{(R\arabic*)}, leftmargin=*, widest=1]
\item The undecorated copy of $r$ (obtained by substituting $1$ for every $\etaleft{1}$) belongs to $\L_c^\infty$. 
\item The inequality $\leftcount(r)< c$ holds.
\item If $\leftcount(r)=k$ and $r=s_0\ \etaleft{1}\ s_1\ \etaleft{1}\ \cdots\ \etaleft{1}\ s_k$ then each factor $s_i$ for $1\le i\le k$ contains a letter of value $c-1$.
\end{enumerate}

Define $\R_c$ to be the language of all words over the alphabet $\mathbb{P}\cup\{\etaleft{1}\}$ satisfying (R1)--(R3). Owing to the fact that the words in $\R_c$ satisfy (R1), it can be seen that the splitting mapping $\psi\st r_j\mapsto (p_j,r_{j+1})$, as defined in Section~\ref{sec-panel-encoding}, can actually be applied to all words in the language $\R_c$. We abuse terminology and denote this extension to $\R_c$ also by $\psi$.

The splitting mapping $\psi\st r_j\mapsto (p_j,r_{j+1})$ (extended to $\R_c$ as described above) produces two words, the second of which also belongs to the language $\R_c$. To characterise the first of these words (the panel word) we introduce another language. We define $\P_c$ to be the set of words $p$ over the alphabet $\{1,2,\dots,c-1\}\cup\{\etaleft{1},\etaright{1},\etaleftright{1}\}$ which satisfy the rules (P1)--(P5) below.
\begin{enumerate}[label=\textup{(P\arabic*)}, leftmargin=*, widest=1]
\item The undecorated copy of $p$ satisfies the small ascent condition.
\item If $p$ is non-empty then its first letter is $1$, $\etaleft{1}$, $\etaright{1}$, or $\etaleftright{1}$.
\item The inequalities $\leftcount(p) < c$ and $\rightcount(p) < c$ hold.
\item Any letter immediately following a right letter of $p$ is one of $1$, $\etaright{1}$, $\etaleft{1}$, or $\etaleftright{1}$.
\item If $\leftcount(p)=k$ and $p=q_0\ \ell_1\ q_1\ \ell_2\ q_2\ \cdots\ q_{k-1}\ \ell_k\ q_k$, where the $\ell_i$ are the left letters of $p$, then each factor $\ell_i q_i$ for $1\le i\le k$ contains either an occurrence of $c-1$ or a right letter.
\end{enumerate}
Before establishing the validity of these properties we make two observations. The first is that not only do the first components of the images of $\psi$ satisfy (P1)--(P5), but also every panel word $p=p_j$ arising in the encodings $\eta_c(w)$ satisfies these conditions. Our second observation is that the language $\P_c$ is clearly regular; to see this it suffices to check that the negation of each of (P1)--(P5) defines a regular language, and then appeal to the closure of regular languages under boolean operations.

Establishing the validity of the properties (P1)--(P5) is fairly straightforward, so we limit ourselves to a few words of justification. For (P1) note that any panel word is a concatenation of factors beginning with $1$ satisfying the small ascent condition, hence does so itself. The property (P3) follows from (R2), because the left letters of $p=p_j$ are inherited from the remainder word $r_j$, while the right letters are matched with the left letters of the remainder word $r_{j+1}$. For (P4), the factor following a right letter up to the next 1 is carried forward to the next remainder, so the next letter of a panel word must have value $1$.  Finally, for (P5), the left letters in $p=p_j$ correspond to the left letters of $r_j$. These, in turn, correspond to the distinguished occurrences of $12\cdots c$ in $r_{j-1}$: of each such occurrence, $\etaleft{2}3\cdots c$ is carried forward into $r_j$ where it becomes $\etaleft{1}2\cdots (c-1)$. If the symbol $c-1$ does not make it from $r_j$ into $p$, the reason is that it is carried forward into $r_{j+1}$, in which case a right letter remains in $p$ to indicate the location of its removal.

Before our next result we need to recall the inverse, $\psi^{-1}$, of the splitting mapping defined in the previous section. As this mapping is simply a formal reversal of the procedure used to compute $\psi$, it can be extended verbatim to all pairs $(p,r)\in \P_c\times\R_c$ that satisfy $\rightcount(p)=\leftcount(r)$.

\begin{proposition}
\label{prop-splitting-bijection}
The extended mappings $\psi$ and $\psi^{-1}$ are mutually inverse bijections between $\R_c$ and
$\{(p,r)\in\P_c\times\R_c\st \rightcount(p)=\leftcount(r)\}$.
\end{proposition}
\begin{proof}
If $s \in \R_c$ and $\psi(s) = (p, r)$, it is easy to see that $(p,r)\in\P_c\times\R_c$. Also, we have $\rightcount(p)= \leftcount(r)$, because at the stage when the letters of $s$ are decorated, the newly decorated letters occur in adjacent pairs and indicate precisely the positions where the splitting into $p$ and $r$ occurs.

Now take an arbitrary pair $(p,r)\in\P_c\times\R_c$ with $\rightcount(p)=\leftcount(r)=k$ and set $s=\psi^{-1}(p,r)$. To establish that $s\in\R_c$ we observe that it must possess the following properties.
\begin{itemize}
\item It lies in $\left(\mathbb{P}\cup\{\etaleft{1}\}\right)^\ast$ because all right arrows have been removed.
\item The undecorated version of $s$ satisfies the small ascent condition because the undecorated copies of $p$ and $r$ satisfy this condition by (P1) and (R1), and the factors inserted into $p$ to form $s$ create $12$ factors at their left hand ends (by definition of $\psi^{-1}$) , and descents at their right hand ends (by (P4)).
\item It satisfies $\leftcount(s)< c$, because the left letters are inherited from those of $p$, which satisfies (P3). 
\item Each factor of $s$ between two consecutive left letters contains an occurrence of $c-1$, as does the suffix following the last left letter. This follows because left letters of $s$ are inherited from those of $p$. Thus by (P5) either there is already an occurrence of $c-1$ in such a factor, or there was a right letter in the corresponding part of $p$. In the latter case, a factor of $r$ beginning with $\etaleft{1}$ was inserted (and increased by $1$) following such a right letter, and by (R3) this results in an occurrence of $c$. The small ascent condition then also guarantees an occurrence of $c-1$.
\end{itemize}
Therefore $s$ indeed satisfies (R1)--(R3), as required. Moreover, $\psi^{-1}$ has been designed precisely so that the splitting mapping reverses it, which completes the proof.
\end{proof}

We now turn to the language $\L_c^\eta$ of all $\eta_c$ encodings of words from $\L_c^\infty$. A typical word $e \in \L_c^\eta$ may be written as $e=p_0\#p_1\#\cdots\#p_{m-1}\#$, where the $p_j$ do not contain $\#$, and satisfies the following conditions.
\begin{enumerate}[label=\textup{(L\arabic*)}, leftmargin=*, widest=1]
\item For all $0\le j<m$ we have $p_j\in\P_c$.
\item For all $0\le j<m-1$ we have $\rightcount(p_j)=\leftcount(p_{j+1}) < c$, and also $\leftcount(p_0)=\rightcount(p_{m-1})=0$.
\item There exists an index $j$ such that the word $p_j$ contains the letter $m-j$.
\item For all $0\le j\le m-1$, no letters of value greater than $m-j$ occur in $p_j$.
\end{enumerate}

Only the last two conditions require comment and they hold because if $e = \eta_c(w)$ then the number of panel words, $m$, in $e$ is equal to the maximum value occurring in $w$. This value must be encoded in some panel, say $p_j$, where it is encoded as $m-j$ by Observation~\ref{obs-where-encoded}, satisfying (L3). No $p_j$ can contain a letter of value greater than $m-j$ because this would correspond to a letter of value greater than $m$ in $w$, showing that (L4) is satisfied.

\begin{proposition}
\label{L1-L4-equals-Lcs}
The language $\L_c^\eta$ consists precisely of all words
\[
	e=p_0\#p_1\#\cdots\# p_{m-1}\#
\]
satisfying {\normalfont (L1)--(L4)}.
\end{proposition}
\begin{proof}
Suppose that $e=p_0\#p_1\#\cdots\# p_{m-1}\#$ satisfies (L1)--(L4). Set $r_m = \epsilon$, and then for $k$ from $m-1$ down to $0$ let $r_{k} = \psi^{-1}(p_k, r_{k+1})$. It follows from Proposition~\ref{prop-splitting-bijection} that, at each step, the conditions required for $\psi^{-1}$ to be defined on the given arguments apply, and so we obtain a sequence $r_{m-1},\dots,r_1,r_0$ of elements of $\mathcal{R}_c$. By (L2), the final word $r_0=w$ does not contain any decorated letters, and so in fact $w\in\mathcal{L}_c^\infty$ by (R1). Furthermore, the conditions (L3) and (L4) imply, via a straightforward inductive argument, that the maximum value occurring in $w$ is precisely $m$. It follows that the panel encoding $\eta_c(w)$ will contain precisely $m$ panel factors. These panel factors are obtained starting from $w=r_0$ and successively applying the splitting mapping $m-1$ times. By Proposition~\ref{prop-splitting-bijection}, the sequence of the remainder words obtained will be precisely $r_0,r_1,\dots,r_{m-1}$, while the sequence of the panel words will be $p_0,\dots,p_{m-1}$, so that $e=\eta_c(w)\in\mathcal{L}_c^\eta$, completing the proof.
\end{proof}

We conclude this section with its main result.

\begin{proposition}
\label{prop-regularity-encoding}
For every positive integer $c$, the language $\L_c^\eta$ is regular.
\end{proposition}
\begin{proof}
We show that the languages defined by the individual conditions (L1)--(L4) are regular. The first follows easily because (L1) defines the language $\left( \P_c \# \right)^\ast$, which is regular because $\P_c$ is regular.

Condition (L2) also defines a regular language because of the bound on the values to be compared. Note that (L2) is violated exactly if there is some panel word containing more than $c$ right letters, or some pair of consecutive panel words where the number of left letters in the second panel word is not the same as the number of right letters in the first. Such a violation is easily recognised by a non-deterministic automaton (i.e., the automaton we describe accepts all words which fail this condition). The automaton accepts (i.e. identifies a violation) if it detects any left letters in the first panel word $p_0$. Otherwise it idles until it reaches some punctuation symbol. Then it counts right letters in the next panel word, accepting if that count exceeds $c$. If it still has not accepted, the automaton remembers this count (which is bounded by $c$) and proceeds to count left letters in the following panel word, again accepting if that does not match the stored count of right letters (or if the stored count is non-zero and there is no following panel word). Once the word has been completely read, if it has not been accepted it is rejected. If a violation occurs, the input word is accepted by some computation of this automaton, while if no violation occurs, no computation accepts the input word. Thus the words satisfying (L2) are the complement of the language accepted by this automaton, and so form a regular language.

Finally note that conditions (L3) and (L4) only present non-vacuous restrictions for the final $c-1$ panel words since $p_j\in\P_c$ for all indices $j$. We verify both conditions with a common automaton which reads encodings from right to left (recall that the reverse of a regular language is also regular). This automaton records the set of letters occurring in each of the panel words $p_{m-1},\dots,p_{m-c+1}$. Since this is a bounded amount of information, it can be stored in a state, and each condition implied by (L3) and (L4) is tested by direct inspection of the recorded information (i.e. by designating the appropriate states as accepting).
\end{proof}

\section{Marking, Transducing, and Greediness}

We have established that $\eta_c$ gives a bijection between $\L_c^{\infty}$ and $\L_c^{\eta}$ (Observation~\ref{obs-where-encoded}), and that the latter language is regular (Proposition~\ref{prop-regularity-encoding}). However, $\mathcal{L}_c^\eta$ is not sufficient for our counting purposes, because a permutation $\pi\in\Av(321)$ generally has several (and a variable number of) possible griddings, and it is the latter that are encoded in $\mathcal{L}_c^\eta$. We therefore need to pass to our distinguished, unique---i.e. greedy---griddings. In other words, we need to consider the set $\eta_c(\mathcal{L}_c^\infty\cap \mathcal{G}^\infty)$ and prove that it is regular. To do this we return to the domino encoding. In general, as noted previously, the domino encoding is not a suitable device for detecting regularity because of the consistency requirement between consecutive dominoes and the lack of bounds on the number of symbols in a domino. Fortunately, the properties we are interested in (initially, greediness; in the next section, finite bases; after that, well-quasi-order) depend only on a bounded number of letters per domino factor. Here we develop a technique, called \emph{marking}, that allows us to focus on such bounded sets of letters.

In a \emph{marked} permutation some of the entries, designated with overlines, are distinguished from the remaining entries. Generally the reason for adding marks to a permutation is to follow the marking with a test that identifies the presence or absence of some specific configuration among the marked elements. Notationally, marked permutations and sets of such permutations are indicated with overlines.

Because our encodings are entry-to-entry mappings or nearly so (in the case of the domino encoding which maps a single entry to two letters), it is easy to define marked versions of them (which we also distinguish with overlines): the encoding of a marked permutation is obtained by marking the letter(s) of the encoding that correspond to marked entries of the permutation. Essentially we double the size of the alphabet, introducing a marked version of each non-punctuation letter. For instance, the \emph{marked omnibus encoding} $\overline{\omega}$ maps marked gridded permutations to words whose letters are either marked or unmarked positive integers. The \emph{marked domino encoding} $\overline{\delta}$ similarly maps marked gridded permutations to $\{\oldpoint,\newpoint,\overline{\oldpoint},\overline{\newpoint},\#\}^\ast$.

We denote by $\overline{\L}_{c}^\eta$ the marked version of $\L_{c}^\eta$, i.e., the set of all marked words which would lie in $\L_{c}^\eta$ if the markings on their non-punctuation symbols were removed. Note that in these words letters can be both decorated (with arrows) and marked (with overlines). Fortunately, we have no need to actually depict this.

Typically we consider markings of gridded permutations such that a bounded number of entries in each cell are marked and then ask about the subpermutation formed by the marked entries. We begin with a simple example of the type of results we establish.

In Section~\ref{sec-omnibus}, we defined domino factors of arbitrary words in $\mathbb{P}^\ast$. Here we extend this definition to arbitrary marked words in $\overline{\mathbb{P}}^\ast$, though we are interested only in the marked letters: given $\overline{w}\in \overline{\mathbb{P}}^\ast$, the \emph{$i\th$ domino factor corresponding to its marked letters} is defined as $\overline{d}_i=\eval{\overline{w}}^{\{\overline{1},\overline{i+1}\}}$. Note that unmarked letters do not occur in $\overline{d}_i$.

\begin{proposition}
\label{prop-marked-domino}
Let $\overline{w}\in\overline{\L}_c^\infty$. The $i\th$ domino factor corresponding to the marked letters of $\overline{w}$ is completely determined by the subword of $\overline{\eta}_{c}(\overline{w})$ consisting of those letters in panel factors $\overline{p}_{i-c+1}$, $\overline{p}_{i-c+2}$, $\dots$, $\overline{p}_{i}$ that are marked or decorated (or both), along with the punctuation symbols between them.
\end{proposition}
\begin{proof}
By Observation~\ref{obs-where-encoded}, the letters $i$ and $i+1$ (and their marked versions) may only be encoded in the panel words $\overline{p}_{i-c+1}$, $\overline{p}_{i-c+2}$, $\dots$, $\overline{p}_{i}$. The result now follows immediately from Proposition~\ref{prop-delete-panels}.
\end{proof}

In fact we need stronger results than that above. We want to translate one encoding into another, restricting to marked entries. For this we use transducers. A \emph{transducer} is a finite-state automaton (not necessarily deterministic) that may produce output while reading. Thus given an input alphabet $\Sigma$ and an output alphabet $\Gamma$, each transition of a transducer has both an input symbol $a\in\Sigma\cup\{\epsilon\}$ and an output symbol $b\in\Gamma\cup\{\epsilon\}$. If the transducer $T$ has an accepting computation on reading $w$, then the output of that computation is the word formed by concatenating the output symbols associated with the transitions performed (in the same order as those transitions). No output is associated with non-accepting computations. Note that output is associated to a specific computation, so for non-deterministic transducers the same input word $w$ may yield multiple outputs. 

A simple and illustrative example is the transducer with input alphabet $\Sigma$ and output alphabet $\overline{\Sigma}$ which marks precisely one letter of its input. This transducer can be defined using an underlying automaton defined by the following three properties.
\begin{itemize}
\item It has an initial non-accepting state that has transitions to itself whose input/output pairs are $a/a$ for each $a \in \Sigma$.
\item It has a second state, which is accepting, that also has transitions to itself whose input/output pairs are $a/a$ for each $a \in \Sigma$.
\item There are transitions from the first to the second state whose input/output pairs are $a/\overline{a}$ for each $a \in \Sigma$.
\end{itemize}

We use functional notation, so if $T$ is a transducer and $X$ is a set of words (of the appropriate alphabet for $T$) then $T(X)$ is the set of words output by $T$ while reading the words of $X$ (which could be empty if none of the words of $X$ are accepted by the underlying automaton). As usual, when $X$ is a singleton we generally omit set braces and write $T(w)$. We utilise the following basic facts about transducers.
\begin{itemize}
\item
If $X$ is a regular language and $T$ is a transducer, then $T(X)$ is again a regular language.
\item
Conversely, if $Y$ is a regular language then the preimage $T^{-1}(Y)=\{x \st T(x)\cap Y\neq \emptyset\}$ is regular as well. 
\item
The composition of two transducers is again a transducer. 
\end{itemize}
For further details see, for example, Sakarovitch~\cite[Chapter IV]{Sakarovitch:Elements-of-aut:}.

For our next result we must make another definition. Given a marked word $\overline{w}\in\overline{\L}^\infty$, the \emph{domino encoding of the word formed by its marked letters} is
\[
	\overline{d}_0^{\newpoint}\#\overline{d}_1^{\newpoint}\#\cdots\#\overline{d}_m^{\newpoint}\#,
\]
where $m$ is the maximum value of a marked or unmarked letter of $\overline{w}$, each $\overline{d}_i$ is the $i\th$ domino factor corresponding to the marked letters of $\overline{w}$ defined previously, and $\overline{d}_i^{\newpoint}$ is the translation of $\overline{d}_i$ to the alphabet $\{\overline{\oldpoint}, \overline{\newpoint}\}$ formed by replacing $\overline{i}$ by $\overline{\oldpoint}$ and $\overline{i+1}$ by $\overline{\newpoint}$.

\begin{proposition}
\label{prop-panel-to-domino}
For every fixed integer $k$ there is a transducer that, given the panel encoding $\overline{\eta}_c(\overline{w})$ of a marked word $\overline{w}\in\overline{\L}_c^\infty$, where $\overline{w}$ contains at most $k$ marked copies of each symbol, outputs the domino encoding of the word formed by the marked letters of $\overline{w}$. Moreover, the output domino encoding has at most $k$ entries in each of its cells.
\end{proposition}
\begin{proof}
Given a panel word, $\overline{p}$, its \emph{stripped form} is the subword consisting of all marked or decorated letters. Since the bound on the number of marked copies of any symbol implies a bound on the number of marked entries in each panel word, and the number of decorated entries in a panel word is bounded in any case, there is a finite set of stripped panel words that can arise from the panel encodings $\overline{\eta}_c(\overline{w})$. We view this set as a new alphabet. We then transduce $\overline{\eta}_c(\overline{w})$ into the word over this alphabet determined by replacing each panel word by the single letter corresponding to its stripped form, deleting (i.e., not transcribing) the punctuation symbols as we proceed. We call the resulting word the \emph{stripped form} of $\overline{\eta}_c(\overline{w})$.

Given an arbitrary alphabet $\Sigma$, a positive integer $c$, and a placeholder symbol $\bullet$ not in $\Sigma$, we form the alphabet $\Gamma=\left(\Sigma \cup \{ \bullet \}\right)^c$ and a transducer from $\Sigma^*$ to $\Gamma^*$ that maps $u \in \Sigma^*$ to a word $v$ in $\Gamma^*$ of the same length with $v(i) = (u(i-c+1), u(i-c+2), \dots, u(i))$ (replacing references to symbols of negative index by $\bullet$). Applying this transducer to the stripped form of $\overline{\eta}_c(\overline{w})$ gives a word whose symbols correspond to the sequences of $c$ consecutive stripped panel words of $\overline{\eta}_c(\overline{w})$. Proposition~\ref{prop-marked-domino} shows that the stripped forms of the marked panel words $\overline{p}_{i-c+1}, \dots, \overline{p}_{i}$ determine the  $i\th$ domino factor for the marked letters of $\overline{w}$, so one final transducer that replaces each such sequence by its corresponding domino factor completes the process.
\end{proof}

%

Up to this point we have worked primarily with $\L_c^{\eta}$, a regular language that is in one-to-one correspondence with $\L_c^\infty$, the language of words $w \in \mathbb{P}^\ast$ that satisfy the small ascent condition and contain no shift of $(12 \cdots c)^c$. Recall that $\G^{\infty}$ is the image of the greedy griddings of $321$-avoiding permutations under the omnibus encoding $\omega$. We define two additional languages:
\[
	\G_c^{\infty}=\G^\infty\cap\L_c^{\infty}
	\quad\text{and}\quad
	\G_c^{\eta}=\eta_c(\G_c^{\infty}).
\]
It is our principal goal in this section to prove that $\G_c^{\eta}$ is regular, i.e., that the panel encodings of omnibus encodings of greedy staircase griddings can be recognised by a finite automaton. By Observation~\ref{obs-greedy-omegaG} and the results of the previous section, this is equivalent to showing that the set of $\eta_c$ encodings of words in $\L_c^\infty$ that satisfy ($\omega$G1) and ($\omega$G2) can be recognised by a finite automaton. Note that these two conditions apply only to the first and last occurrence of each letter. Furthermore, the first (resp., last) occurrence of each letter in a word $w\in\L^\infty$ will also be the first (resp., last) occurrence of the corresponding letter in some panel word of $\eta_c(w)$. Therefore our first step is to describe a transducer which marks the first and last letter of each value in every panel word of $\eta_c(w)$.

\begin{proposition}
\label{prop-mark-first-and-last}
There is a transducer that, given $w \in \L_c^{\eta}$, outputs a marked panel encoding $\overline{w}$ in which the first and last entries of each value in each panel word are marked.
\end{proposition}
\begin{proof}
It suffices to define the operation of such a transducer on a single panel word---the full transducer can then be built by returning to the initial state when a punctuation symbol is read. In turn it suffices to construct such a transducer for each individual value $k$ of a letter from 1 through $c-1$ (since these can then be composed to give the required transducer). The transducer defined by the following properties performs this task.
\begin{itemize}
\item
The initial state, \textsf{start}, is an accepting state.
\item
In any state the transducer transcribes all input that is not a $k$ (that is, outputs the same symbol as the input symbol) and remains in the current state.
\item
When (or if) the transducer first encounters a $k$, it outputs $\overline{k}$ and enters either state \textsf{seenfirst} or \textsf{seenlast} (non-deterministically).
\item
In state \textsf{seenfirst} (which is non-accepting) if the transducer encounters a $k$ it either transcribes it and remains in state \textsf{seenfirst}, or outputs $\overline{k}$ and enters state \textsf{seenlast}.
\item
In state \textsf{seenlast} (which is accepting) if the transducer encounters a $k$ then it fails, resulting in no output (this can be implemented by way of a state \textsf{fail} which has no further transitions).
\end{itemize}
Note that in the case $k=1$, some of the occurrences of $1$ in the input word may be decorated with arrows---the transducer retains those arrows as well as possibly adding marking.
\end{proof}

Propositions~\ref{prop-panel-to-domino} and \ref{prop-mark-first-and-last} give us the machinery we need in order to verify compliance with conditions ($\omega$G1) and ($\omega$G2), allowing us to prove the main result of the section.

\begin{proposition}
\label{prop-greediness-permutations}
For every positive integer $c$, the language $\G_c^{\eta}$ is regular.
\end{proposition}
\begin{proof}
We know that $\L_c^{\eta}$ is regular by Proposition~\ref{prop-regularity-encoding}. We prove the result by showing that $\G_c^{\eta}$ is the intersection of $\L_c^{\eta}$ and the preimage of a regular language under a transducer.

Given an encoding $\eta_c(w)\in\L_c^{\eta}$, we first pass it through the transducer of Proposition~\ref{prop-mark-first-and-last} to mark the first and last occurrences of each letter in every panel. By Observation~\ref{obs-where-encoded}, the letter $i$ in $w$ corresponds to an entry of one of the $c-1$ panel words $p_{i-c+1}$ to $p_{i-1}$, so by marking the first and last occurrences of each letter in every panel, we have marked at most $k=2(c-1)$ copies of each symbol. We then pass these marked encodings through the transducer of Proposition~\ref{prop-panel-to-domino} to obtain the domino encoding of these marked letters. Let $T$ denote this composition of transducers.

Note that $T$ produces precisely one output for each $w\in\L_c^{\eta}$, so we denote this output by $T(w)$, temporarily neglecting our convention that transducers output sets. Also note that the word $T(w)$ contains, in addition to other letters, letters corresponding to the first and last occurrences of each symbol in $w$. Thus $T(w)$ contains enough information to allow us to decide whether $w$ satisfies the conditions ($\omega$G1) and ($\omega$G2), though we still must show how to make this determination in a regular manner.

As observed above, every domino factor in $T(w)$ has at most $2(c-1)$ occurrences of each letter (a first and last occurrence of the letter in all $c-1$ panel words it could be encoded in). Thus there is a finite set, say $\Delta$, of dominoes which occur in the domino encodings output by $T$. We may therefore (by means of a simply transducer we will not describe explicitly) consider $\Delta$ itself to be the output alphabet and ignore the punctuation symbols (which are superfluous at this point), so that $T(w)\in \Delta^\ast$ for all $w\in\L_c^{\eta}$. 

Thus we may assume that we are given the actual sequence of dominos which corresponds to the word $T(w)$, and we must describe how to determine whether $T(w)$, or equivalently, $w$ itself, satisfies ($\omega$G1) and ($\omega$G2). These conditions translate to simple conditions on the dominoes of $T(w)$: each domino factor other than the first must begin with $\oldpoint$, and each domino factor other than the first and last must contain the subword $\newpoint\oldpoint$. Let $\mathcal{R}\subseteq\Delta^\ast$ denote the language of domino encodings which satisfy these conditions. Clearly $\mathcal{R}$ is regular, and it follows that $\mathcal{G}_c^{\eta}=\mathcal{L}_c^{\eta}\cap T^{-1}(\mathcal{R})$, completing the proof.
\end{proof}

We claim that $\G_c^{\eta}$ is the intersection of $\L_c^{\eta}$ (which is regular by Proposition~\ref{prop-regularity-encoding}) and $T^{-1}(\mathcal{R})$, where $\mathcal{R}$ is a regular language.

\section{Detecting Basis Elements}
\label{sec-detecting-basis-elts}


The results of the previous sections establish that, for each positive integer $c$, the set of $321$-avoiding permutations such that the omnibus encodings of their greedy griddings do not contain any shift of $(12 \dots c)^c$ is in bijective correspondence with the regular language $\G_c^{\eta}$. We have also observed in Proposition~\ref{prop-avoid-shifts} that for any proper subclass $\C\subsetneq\Av(321)$ there is a positive integer $c$ such that $\omega(\pi^\gridded) \in \G_c^{\infty}$ for all greedy griddings $\pi^\gridded$ of permutations $\pi\in\C$ and hence the panel encodings of these omnibus encodings are contained in $\G_c^{\eta}$. To complete our goal of showing that any such \emph{finitely based} class has a rational generating function, we need to show how to detect avoidance (or, equivalently, containment) of specified permutations within the panel encodings, while maintaining regularity.

The difficulty we are facing is that none of the three encodings we have used thus far---the omnibus encoding, its composition with the panel encoding, and the domino encoding---provide an easy way to test containment (as discussed briefly after Observation~\ref{obs-contain-omnibus-matchings}). To overcome this difficulty we resort again to the technique of marking, but this time we transduce the marked subpermutation to yet another encoding, namely the Dyck path encoding. This encoding---which was essentially described in the Introduction and illustrated on the right of Figure~\ref{fig-bij-Dyck}---consists of constructing a Dyck path whose outer corners lie just outside the left-to-right maxima of the permutation. We turn the resulting Dyck paths into words over the alphabet $\{ \textsf{u},\textsf{d}\}$ in the standard way. For instance, the Dyck path encoding of the permutation $31562487$ depicted in Figure~\ref{fig-bij-Dyck} is $\textsf{u}^3\textsf{d}^2\textsf{u}^2\textsf{d}\textsf{u}\textsf{d}^3\textsf{u}^2\textsf{d}^2$.

\begin{proposition}
\label{prop-domino-to-Dyck}
For every fixed positive integer $k$ there is a transducer that, given the domino encoding of a staircase gridding of a $321$-avoiding permutation $\pi$ with at most $k$ entries per cell, outputs the Dyck path corresponding to $\pi$.
\end{proposition}
\begin{proof}
As in the proof of Proposition~\ref{prop-panel-to-domino} the bound on the number of entries per cell means that we may view the domino factors as letters themselves coming from a finite alphabet. In fact, borrowing another idea from the same proposition, we can view triples of consecutive translated domino factors $d_{2i-1}^{\newpoint}$, $d_{2i}^{\newpoint}$, $d_{2i+1}^{\newpoint}$ (including padding at the beginning and end by empty domino factors) as individual letters. The reason for doing this is that we will show that we can compute the part of the Dyck path determined by the left-to-right maxima lying in the $2i\th$ and $(2i+1)\first$ cells from such a triple. Thus our transducer need only examine these triples in turn, and output the appropriate segment of a Dyck path for each one. This is illustrated in Figure~\ref{fig-domino-to-Dyck}. 

\begin{figure}
\[
\begin{array}{ccccc}
	\newpoint\oldpoint\oldpoint\newpoint\newpoint\newpoint
	&
	\oldpoint\oldpoint\newpoint\newpoint\oldpoint\newpoint\oldpoint
	&
	\oldpoint\oldpoint\newpoint\newpoint\oldpoint
	\\[4pt]
    \begin{tikzpicture}[scale=.25, baseline=0]
	    \plotpartialpermhollow{1/2,2/3};
	    \plotpartialperm{3/1,4/4,7/5,9/6};
		\plotpermbox{1}{1}{2}{6}; 
		\node [above left] at (2.75,0.35) {\scriptsize $1$};
		\plotpermbox{3}{1}{9}{6}; 
		\node [above left] at (9.75,0.35) {\scriptsize $2$};
		\draw [->, thick] (-0.1,0.5)--(-0.1,6.5);
    \end{tikzpicture}
&
    \begin{tikzpicture}[scale=.25, baseline=0]
	    \plotpartialpermhollow{3/1,4/4,7/5,9/6};
	    \plotpartialperm{5/7,6/8,8/11};
		\plotpermbox{3}{1}{9}{6}; 
		\node [above left] at (9.75,0.35) {\scriptsize $2$};
		\plotpermbox{3}{7}{9}{11}; 
		\node [above left] at (9.75,6.35) {\scriptsize $3$};
		\draw [->, thick] (2.5,-0.1)--(9.5,-0.1);
    \end{tikzpicture}
&
    \begin{tikzpicture}[scale=.25, baseline=0]
	    \plotpartialpermhollow{5/7,6/8,8/11};
	    \plotpartialperm{10/9,12/10};
		\plotpermbox{3}{7}{9}{11}; 
		\node [above left] at (9.75,6.35) {\scriptsize $3$};
		\plotpermbox{10}{7}{13}{11}; 
		\node [above left] at (13.75,6.35) {\scriptsize $4$};
		\draw [->, thick] (1.9,6.5)--(1.9,11.5);
    \end{tikzpicture}
&
	\begin{tikzpicture}[scale=.25, baseline=0]
		\draw [white] (0,0.5)--(0,11.5);
		\node at (0,6) {$\longrightarrow$};
	\end{tikzpicture}
&
	\begin{tikzpicture}[scale=.25]
		\plotpartialperm{1/2,2/3,3/1,4/4,5/7,6/8,7/5,8/11,9/6,10/9,12/10};
		\plotpermbox{1}{1}{2}{6}; 
		\node [above left] at (2.75,0.35) {\scriptsize $1$};
		\plotpermbox{3}{1}{9}{6}; 
		\node [above left] at (9.75,0.35) {\scriptsize $2$};
		\plotpermbox{3}{7}{9}{11}; 
		\node [above left] at (9.75,6.35) {\scriptsize $3$};
		\plotpermbox{10}{7}{13}{11}; 
		\node [above left] at (13.75,6.35) {\scriptsize $4$};
		\draw [very thick, rounded corners=0.01, line cap=round] (2.5,3.5)--(3.5,3.5)--(3.5,4.5)--(4.5,4.5)--(4.5,7.5)--(5.5,7.5)--(5.5,8.5)--(7.5,8.5)--(7.5,11.5)--(9.5,11.5);
    \end{tikzpicture} 
\\[4pt]
d_1&d_2&d_3
\end{array}
\]
\caption{Upon reading the triple of domino factors shown in the top left, the transducer of Proposition~\ref{prop-domino-to-Dyck} can compute the partial permutation shown on the right, and output the steps of the Dyck path passing through the $2\nd$ and $3\rd$ cells, $\textsf{d}\textsf{u}\textsf{d}\textsf{u}\textsf{u}\textsf{u}\textsf{d}\textsf{u}\textsf{d}\textsf{d}\textsf{u}\textsf{u}\textsf{u}\textsf{d}\textsf{d}$.}
\label{fig-domino-to-Dyck}
\end{figure}
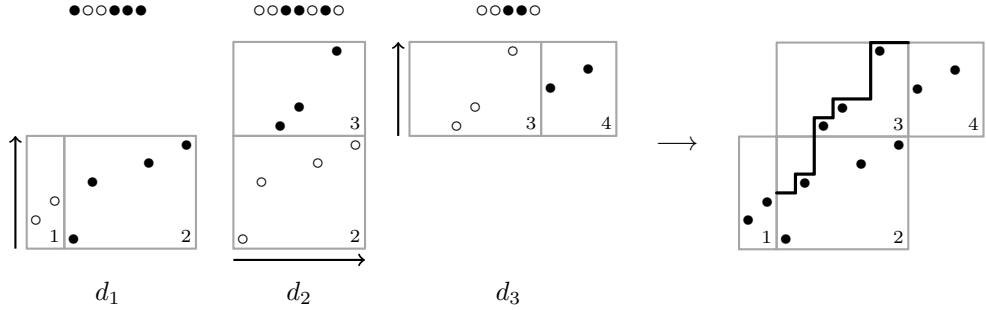

To justify the claim we note that the information encapsulated in $d_{2i-1}^{\newpoint}$, $d_{2i}^{\newpoint}$, $d_{2i+1}^{\newpoint}$ completely determines the relative values and positions of all entries in the $(2i-1)\first$ through $(2i+2)\nd$ cells. In particular it determines the left to right maxima in the $(2i)\th$ and $(2i+1)\first$ cells, and their relative positions with respect to the entries to the right and below them, all of which can be found in the $(2i)\th$ and $(2i+2)\nd$ cells. The final entry in the $(2i-1)\first$ cell (which is automatically a left to right maximum) indicates the entry point of the Dyck path into the $(2i)\th$ cell. (If the $(2i-1)\first$ cell is empty the path enters through the bottom left corner.) From the entry point, the path proceeds as dictated by the left to right maxima in the $(2i)\th$ and $(2i+1)\first$ cells and the entries to the right and below them.
\end{proof}

For any $\beta \in \Av(321)$ and positive integer $c$ we now define $\G_{c, \ge\beta}^{\eta}$ to be the set of all encodings $\eta_c(\omega(\pi^\gridded))$ such that
\begin{itemize}
\item $\pi^\gridded$ is the greedy encoding of $\pi$,
\item $\pi$ contains $\beta$, and
\item $\omega(\pi^\gridded)$ avoids all shifts of $(12\cdots c)^c$.
\end{itemize}
The transducer from our previous proposition shows that this is a regular language:

\begin{proposition}
\label{prop-single-base-regular}
The language $\G_{c, \ge\beta}^{\eta}$ is regular.	
\end{proposition}
\begin{proof}
Let $k$ denote the length of $\beta$. There is a non-deterministic transducer that takes words in $\L_{c}^{\eta}$ as input and outputs marked versions that contain exactly $k$ marked letters.  Denote by $T$ the composition of that transducer and the one defined in Proposition~\ref{prop-panel-to-domino} (which allows for up to $k$ copies of each symbol) followed by the transducer described in Proposition~\ref{prop-domino-to-Dyck}. Further let $X_{\beta}$ denote the singleton set whose only element is the word over the alphabet  $\{\textsf{u}, \textsf{d}\}$ that represents the Dyck path corresponding to $\beta$. 

Since $T$ takes as input the panel encoding of the greedy gridding of a $321$-avoiding permutation, marks exactly $k$ letters, and outputs the Dyck path encoding of the marked letters, it follows that the panel encoding of the greedy gridding of some permutation $\pi\in\Av(321)$ belongs to $T^{-1}(X_\beta)\cap \G_c^{\eta}$ if and only if $\beta$ is contained in $\pi$. Thus $\G_{c,\ge\beta}^{\eta} = T^{-1}(X_\beta)$ and, being the preimage of a regular language (any singleton is regular) under a transducer, is itself regular.
\end{proof}

%
%

We have finally reached the point where we can prove the first half of our main result.

\newenvironment{proof-of-321-rational-fb}{\medskip\noindent {\it Proof of Theorem~\ref{thm-321-rational} (for finitely based subclasses).\/}}{\qed\bigskip}
\begin{proof-of-321-rational-fb}
Suppose that the basis of a class $\C$ is the finite, nonempty, set $B$. Take any positive integer $c$ such that $\omega(\pi^\gridded)\in \G_c^{\infty}$ for all greedy griddings $\pi^\gridded$ of permutations in $\C$ (the existence of such a value of $c$ is guaranteed by Proposition~\ref{prop-avoid-shifts}). Then the set, $\G_{c,\C}^{\eta}$, of panel encodings of greedy griddings of members of $\C$ is
\[
	\G_{c,\C}^{\eta} = \G_{c}^{\eta} \setminus \bigcup_{\beta \in B} \G_{c, \ge\beta}^{\eta}. 
\]
This is a regular language owing to Propositions~\ref{prop-greediness-permutations} and \ref{prop-single-base-regular} and the closure of the family of regular languages under Boolean operations. Therefore $\C$ is in one-to-one correspondence with a regular language. Moreover, if $\pi \in \C$ has length $n$ then its image under the correspondence contains $n$ non-punctuation symbols. The generating function of a regular language over commuting variables corresponding to its letters is a rational function and we can obtain the generating function for $\C$ from that for $\G_{c,\C}^{\eta}$ by replacing the variable corresponding to the punctuation symbol $\#$ by $1$, and those variables corresponding to non-punctuation symbols by $x$, so the generating function of $\C$ is rational.
\end{proof-of-321-rational-fb}

\section{Well-Quasi-Ordered Subclasses}
\label{sec-wpo}

It remains to prove the second half of Theorem~\ref{thm-321-rational}, namely that every well-quasi-ordered subclass of $321$-avoiding permutations has a rational generating function. This proof breaks naturally into two parts. First we identify a necessary and sufficient condition for a subclass of $\Av(321)$ to be well-quasi-ordered. Then we show, using arguments similar to those in the preceding section, that this condition implies regularity of the corresponding languages. For the first part we identify a particular infinite antichain $U \subseteq \Av(321)$. Obviously, for a class $\C \subseteq \Av(321)$ to be well-quasi-ordered, $\C \cap U$ must be finite. It happens that this condition is also sufficient. We begin with some preparatory remarks.

A permutation $\pi$ is said to be \emph{sum decomposable} if it can be written as a concatenation $\alpha\beta$ where every entry in the prefix $\alpha$ is smaller than every entry in the suffix $\beta$. If $\pi$ has no non-trivial partition of this form then it is said to be \emph{sum indecomposable}. We may in this way interpret an arbitrary permutation as a word over its sum indecomposable components (\emph{sum components} for short).

Moving to a more general context, given a poset $(P,\le)$, the \emph{generalised subword order} on $P^\ast$ is defined by $v\le w$ if there are indices $1\le i_1<i_2<\cdots<i_{|v|}\le|w|$ such that $v(j)\le w(i_j)$ for all $j$. The following well-known result connects the well-quasi-ordering of $P$ and $P^\ast$.

\newtheorem*{higmans-lemma}{Higman's Lemma~\cite{higman:ordering-by-div:}}
\begin{higmans-lemma}
If $(P,\le)$ is well-quasi-ordered then $P^*$, ordered by the generalised subword order, is also well-quasi-ordered.
\end{higmans-lemma}

Returning to the context of permutations (and the containment order defined on them), Higman's Lemma easily implies the following result. (For more details we refer the reader to Atkinson, Murphy, and Ru\v{s}kuc~\cite[Theorem 2.5]{atkinson:partially-well-:}.)

\begin{proposition}
\label{prop-wqo-sum-components}
Let $\C$ be a permutation class. If the sum indecomposable members of $\C$ are well-quasi-ordered, then $\C$ is well-quasi-ordered.
\end{proposition}


The identification of the antichain $U$ requires a short digression related to a connection between permutations and graphs. Given a permutation $\pi$, the \emph{inversion graph} (more commonly called a \emph{permutation graph} in the graph theory literature) corresponding to $\pi$ is the graph $G_\pi$ on the (unlabelled) vertices $\{(i,\pi(i))\}$ in which $(i,\pi(i))$ and $(j,\pi(j))$ are adjacent if they form an inversion, i.e., $i<j$ and $\pi(i)>\pi(j)$. As each entry of $\pi$ corresponds to a vertex of $G_\pi$, we commit a slight abuse of language by referring (for example) to the degree of an entry of $\pi$ when we mean the degree of the corresponding vertex of $G_\pi$. Note that the graph $G_\pi$ is connected if and only if $\pi$ is sum indecomposable.

If $\sigma$ is a subpermutation of $\pi$, then the induced subgraph of $G_\pi$ on the entries corresponding to a copy of $\sigma$ is isomorphic to $G_\sigma$. Thus the image of a permutation class under the mapping $\pi \mapsto G_\pi$ is a class of inversion graphs closed under taking induced subgraphs. More importantly for our purposes, the inverse image of an antichain of graphs (in the induced subgraph ordering) is an antichain of permutations. Note incidentally that this is true even though the mapping $\pi \to G_\pi$ is not injective (in particular, $G_{\pi}\cong G_{\pi^{-1}}$ for all permutations $\pi$).

We begin by considering permutations whose graphs are isomorphic to paths on $n\ge 4$ vertices. By direct construction it is easy to verify that there are precisely two such permutations of each length, which we call \emph{increasing oscillations}:
\[
\begin{array}{ll}
2  4  1  6  3  8  5  \cdots n  (n-3)  (n-1),
\quad
3  1  5  2  7  4  9  \cdots (n-4)  n (n-2)
& \text{if $n$ is even, and} \\
2  4  1  6  3  8  5  \cdots (n-4)  n  (n-2),
\quad
3  1  5  2  7  4  9  \cdots n  (n-3) (n-1)
& \text{if $n$ is odd.} 
\end{array}
\]
Inspection of these permutations shows that inversion graphs cannot contain induced cycles on $5$ or more vertices.  As occurrences of $321$ in permutations correspond to triangles in their inversion graphs, it follows that the $321$-avoiding permutations correspond to bipartite inversion graphs. These graphs have previously been studied in the context of well-quasi-order by Lozin and Mayhill~\cite{lozin:canonical-antic:}, although we do not require their results here.

A \emph{double-ended fork} is the graph formed from a path by adding four vertices of degree one, two adjacent to one end of the path and two adjacent to the other. An example is shown in Figure~\ref{fig-double-ended-fork}. It is clear that the set of double-ended forks is an antichain of graphs in the induced subgraph ordering.

\begin{figure}
\begin{footnotesize}
\begin{center}
	\begin{tikzpicture}[scale=0.4, baseline=(current bounding box.center)]
		\draw (1,-1.5)--(2,0)--(1,1.5);
		\draw (2,0)--(5.5,0);
		\draw (7.5,0)--(11,0);
		\draw (12,-1.5)--(11,0)--(12,1.5);
		\plotpartialperm{1/-1.5,1/1.5,2/0,3/0,4/0,5/0,8/0,9/0,10/0,11/0,12/-1.5,12/1.5};
		\node at (6.5,0) {$\dots$};
	\end{tikzpicture}
\end{center}
\end{footnotesize}
\caption{A double-ended fork.}
\label{fig-double-ended-fork}
\end{figure}
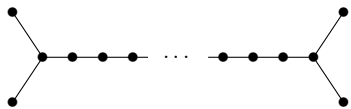


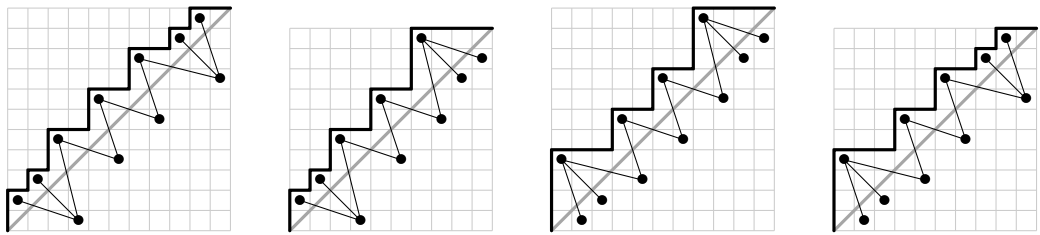
\begin{figure}
\begin{footnotesize}
\begin{center}
	\begin{tikzpicture}[scale=0.26925]
	    \plotpermdyckgrid{2,3,5,1,7,4,9,6,10,11,8};
	    \plotpermgraph{2,3,5,1,7,4,9,6,10,11,8};
		\plotpermdyckpath{1,1,-1,1,-1,1,1,-1,-1,1,1,-1,-1,1,1,-1,-1,1,-1,1,-1,-1};
    \end{tikzpicture}
\quad\quad
	\begin{tikzpicture}[scale=0.26925]
	    \plotpermdyckgrid{2,3,5,1,7,4,10,6,8,9};
	    \plotpermgraph{2,3,5,1,7,4,10,6,8,9};
		\plotpermdyckpath{1,1,-1,1,-1,1,1,-1,-1,1,1,-1,-1,1,1,1,-1,-1,-1,-1};
    \end{tikzpicture}
\quad\quad
	\begin{tikzpicture}[scale=0.26925]
	    \plotpermdyckgrid{4,1,2,6,3,8,5,11,7,9,10};
	    \plotpermgraph{4,1,2,6,3,8,5,11,7,9,10};
		\plotpermdyckpath{1,1,1,1,-1,-1,-1,1,1,-1,-1,1,1,-1,-1,1,1,1,-1,-1,-1,-1};
    \end{tikzpicture}
\quad\quad
	\begin{tikzpicture}[scale=0.26925]
	    \plotpermdyckgrid{4,1,2,6,3,8,5,9,10,7};
	    \plotpermgraph{4,1,2,6,3,8,5,9,10,7};
		\plotpermdyckpath{1,1,1,1,-1,-1,-1,1,1,-1,-1,1,1,-1,-1,1,-1,1,-1,-1};
    \end{tikzpicture}
\end{center}
\end{footnotesize}
\caption{The different types of members of $U$, shown with both their inversion graphs and associated Dyck paths.}
\label{fig-antichains}
\end{figure}

Let $U$ denote the set of all permutations $\pi$ for which $G_\pi$ is isomorphic to a double-ended fork. As in the case of increasing oscillations, direct construction shows that there are four slightly different types of members of $U$, depicted in Figure~\ref{fig-antichains}. By inspection $U\subseteq\Av(321)$, which also follows because double-ended forks are bipartite. By our previous remarks, it follows that $U$ forms an infinite antichain. In particular, every well-quasi-ordered subclass of $\Av(321)$ must have finite intersection with $U$. To establish the other direction, we begin with the following structural result.

\begin{proposition}
\label{prop-bound-paths}
If the subclass $\C\subseteq\Av(321)$ has finite intersection with $U$ then there is a number $\ell$ such that for all connected graphs $G_\pi$ with $\pi\in\C$, the distance between any two vertices of degree three or greater is at most $\ell$.
\end{proposition}
\begin{proof}
Suppose that $\C$ contains no members of $U$ of length $\ell+2$ or longer (here length refers to the length of the permutation) for some $\ell\ge 4$ and choose $\pi\in\C$ to be an arbitrary sum indecomposable permutation.

Let $x$ and $y$ be two entries of $\pi$ of degree three or greater and suppose to the contrary that the distance between these vertices is greater than $\ell$, so there is a shortest path $P$ in $G_\pi$ between $x$ and $y$ with at least $\ell$ internal vertices. Because $x$ and $y$ each have degree at least three, $x$ has neighbours $x_1\neq x_2$ which do not lie on $P$ and $y$ has neighbours $y_1\neq y_2$ which do not lie on $P$. Because the distance between $x$ and $y$ is at least $\ell\ge 4$, note that neither $x_1$ nor $x_2$ can be adjacent to $y$, $y_1$, or $y_2$ (and vice versa with $x$ and $y$ swapped). Also, because $G_\pi$ does not contain a triangle, $x_1$ is not adjacent to $x_2$ and $y_1$ is not adjacent to $y_2$. If none of $x_1$, $x_2$, $y_1$, or $y_2$ are adjacent to any vertices of $P$ other than $x$ or $y$ then $P\cup\{x_1,x_2,y_1,y_2\}$ is isomorphic to a double-ended fork on at least $\ell+6$ vertices (as shown on the left of Figure~\ref{fig-bound-paths}), a contradiction.

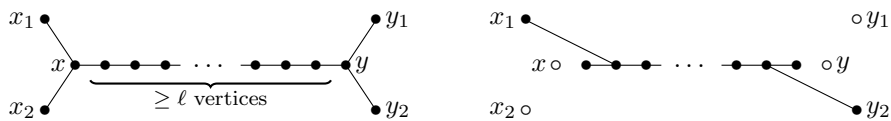
\begin{figure}
\begin{center}
	\begin{tikzpicture}[scale=0.4, baseline=(current bounding box.center)]
		\draw (1,-1.5)--(2,0)--(1,1.5);
		\draw (2,0)--(5.5,0);
		\draw (7.5,0)--(11,0);
		\draw (12,-1.5)--(11,0)--(12,1.5);
		\plotpartialperm{1/-1.5,1/1.5,2/0,3/0,4/0,5/0,8/0,9/0,10/0,11/0,12/-1.5,12/1.5};
		\node at (6.5,0) {$\dots$};
		\node [left] at (1,1.5) {$x_1$};
		\node [left] at (1,-1.5) {$x_2$};
		\node [left] at (2,0) {$x$};
		\node [right] at (12,1.5) {$y_1$};
		\node [right] at (12,-1.5) {$y_2$};
		\node [right] at (11,0) {$y$};
		\draw [thick,decoration={brace, mirror}, decorate] (2.5,-0.5) -- (10.5,-0.5)
			node [pos=0.5,anchor=north] {\text{\footnotesize $\ge\ell$ vertices}};
	\end{tikzpicture}
\quad\quad
	\begin{tikzpicture}[scale=0.4, baseline=(current bounding box.center)]
		\draw (1,1.5)--(4,0);
		\draw (3,0)--(5.5,0);
		\draw (7.5,0)--(10,0);
		\draw (9,0)--(12,-1.5);
		\plotpartialperm{1/-1.5,1/1.5,2/0,3/0,4/0,5/0,8/0,9/0,10/0,11/0,12/-1.5,12/1.5};
		\absdothollow{(1,-1.5)}{};
		\absdothollow{(2,0)}{};
		\absdothollow{(11,0)}{};
		\absdothollow{(12,1.5)}{};
		\node at (6.5,0) {$\dots$};
		\node [left] at (1,1.5) {$x_1$};
		\node [left] at (1,-1.5) {$x_2$};
		\node [left] at (2,0) {$x$};
		\node [right] at (12,1.5) {$y_1$};
		\node [right] at (12,-1.5) {$y_2$};
		\node [right] at (11,0) {$y$};
	\end{tikzpicture}
\end{center}
\caption{Two situations which arise in the proof of Proposition~\ref{prop-bound-paths}.}
\label{fig-bound-paths}
\end{figure}

On the other hand, if one or both of $x_1$ or $x_2$ were adjacent to another vertex of $P$ then it could not be the vertex of $P$ at distance one from $x$ as this would create a triangle (a copy of $321$ in $\pi$) and it also could not be a vertex of distance three or greater from $x$ as this would contradict our choice of $P$ (as a shortest path). Thus the only possibility would be the vertex of $P$ at distance two from $x$, as shown on the right of Figure~\ref{fig-bound-paths}. An analogous analysis implies that if one or both of $y_1$ or $y_2$ were adjacent to another vertex of $P$ then that vertex would have to be the vertex of distance two from $y$. In any case, as indicated on the right of Figure~\ref{fig-bound-paths}, we find an induced double-ended fork on at least $\ell+2$ vertices, a contradiction which completes the proof.
\end{proof}

We are now ready to prove that having finite intersection with $U$ is a sufficient condition for a subclass of $321$-avoiding permutations to be well-quasi-ordered. By Proposition~\ref{prop-wqo-sum-components}, it suffices to consider the sum indecomposable members of our subclass. We then use Proposition~\ref{prop-bound-paths} to show that these sum indecomposable permutations have severely constrained structure; in particular, we show that it implies that ``most'' of their entries are confined to a bounded number of cells. This characterisation is then shown to be sufficient for another appeal to Higman's Lemma, from which well-quasi-ordering follows.

\begin{theorem}
\label{thm-wqo-fin-int}
A subclass $\C \subseteq \Av(321)$ is well-quasi-ordered if and only if $\C \cap U$ is finite.
\end{theorem}
\begin{proof}
By our previous remarks, it suffices to show that if $\C\cap U$ is finite for a subclass $\C\subseteq\Av(321)$ then the sum indecomposable permutations in $\C$ are well-quasi-ordered. To this end, suppose that $\C\cap U$ is finite, choose a sum indecomposable permutation $\pi\in\C$, and fix a particular (not necessarily greedy) staircase gridding $\pi^\gridded$ of $\pi$. Thus every entry of $\pi$ lies in some cell; we refer to the number of this cell as the \emph{label} of the entry or corresponding vertex in $G_\pi$.

Because inversions in $\pi$ can occur only between adjacent cells in the gridding, we conclude that the labels of adjacent vertices in $G_\pi$ differ by precisely $1$. In particular, the distance between two entries of $\pi$ in $G_\pi$ is bounded below by the difference of their labels. Thus by Proposition~\ref{prop-bound-paths}, all vertices in $G_\pi$ of degree three or greater have labels in some bounded interval $\{i,i+1,\dots,i+\ell\}$, where $i$ is the least label of such a vertex (if no such vertices exist, choose $i=0$) and $\ell$ depends only on $\C$. We refer to all entries of $\pi$ in these cells as the \emph{core} of $\pi$.

We aim to partition the entries of $\pi^\gridded$ into three groups: a body, comprising the core of $\pi^\gridded$ together with some of the entries from the adjacent cells at either end, a lower-left tail, and an upper-right tail. The two tails will comprise the entries of $\pi^\gridded$ to the southwest (respectively, northeast) of the core, and the graph induced by each tail will be shown to be a path.

To define this partition, first consider the entries outside the core in $G_\pi$. This set is naturally divided into two pieces: $T_{\text{SW}}$, consisting of entries belonging to cells of label less than $i$, and $T_{\text{NE}}$, consisting of entries belonging to cells of label greater than $i+\ell$. Since all vertices in these pieces have degree at most two, each consists of a disjoint union of paths, and we claim in fact that each piece can contain at most one such path. Suppose to the contrary that the vertices within $T_{\text{SW}}$ were to form at least two disjoint paths. Then (by the nature of edges in inversion graphs), the points in one of these paths would all lie strictly below and to the left of the points of the other path, and thus those points would also lie strictly below and to the left of the points which constitute the core, contradicting our assumption that $\pi$ is sum indecomposable.

Consequently, every vertex of $G_\pi$ that does not correspond to an entry in the core either lies in one of two paths or is only adjacent to (at most two) vertices in the core. This latter collection of vertices must all lie in one of the two cells immediately adjacent to the cells that form the core, and we form the \emph{body} of $\pi$ by adding all these entries to the core (at which point the body is contained in at most $\ell+3$ cells). The entries of $T_{\text{SW}}$ which still lie outside the body now form a path in $G_\pi$. This path, if nonempty, must contain at least two vertices as otherwise it would already be included in the body. If the path is nonempty, we add the vertex of this path which is adjacent to the core to the body and call the remaining vertices the \emph{lower-left tail}. We then perform the analogous operation on the entries of $T_{\text{NE}}$ to form the \emph{upper-right tail}. Note that the body is contained in at most $\ell+3$ cells at the end of this process.

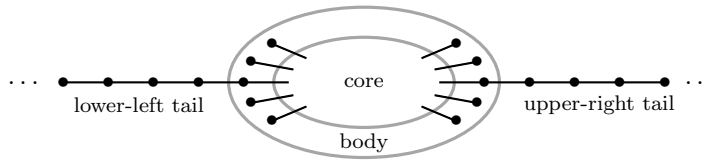
\begin{figure}
\begin{center}
	\begin{tikzpicture}[scale=0.2, baseline=(current bounding box.center)]
		\draw[darkgray, very thick, line cap=round] (0,0) ellipse (9 and 5);
		\draw[darkgray, very thick, line cap=round] (0,0) ellipse (6 and 3);
		\foreach \x in {-8,-11,...,-20}
			\absdot{(\x,0)}{};
		\foreach \x in {8,11,...,20}
			\absdot{(\x,0)}{};
		\foreach \angle in {-20,-40,20,40,140,160,200,220}
			\draw[fill=black, thick] (\angle:8 and 4) node {\small $\bullet$} 
				-- (\angle:5 and 2.5);
		\draw [thick] (5,0)--(20,0);
		\draw [thick] (-5,0)--(-20,0);	
		\node at (0,-4) {\footnotesize body};
		\node at (0,0) {\footnotesize core};
		\node [left] at (-10,-1.5) {\footnotesize lower-left tail};
		\node [right] at (10,-1.5) {\footnotesize upper-right tail};
		\node at (22.5,0) {$\dots$};
		\node at (-22.5,0) {$\dots$};
	\end{tikzpicture}
\end{center}
\caption{The core, body, and tails of a $321$-avoiding inversion graph.}
\label{fig-graph-body-and-tail}
\end{figure}

Our sum indecomposable permutation $\pi$ now has a graph of the form shown in Figure~\ref{fig-graph-body-and-tail} where each of the two tails is either absent or else contains at least one vertex outside the body which is adjacent to a vertex of degree two inside the body. Note also that it is possible in our gridding of $\pi$ that some entries of the two tails can share cells with entries of the body, but this is of no consequence: they are included in the tail, and not in the body.

The subpermutation of $\pi$ that makes up the body of $\pi$, together with the first point of each tail (i.e., the one adjacent to the body, if there is a tail) inherits a staircase gridding (which need not be greedy) from $\pi^\gridded$ in which it occupies not more than $\ell+5$ cells. This means that the body has a gridding into cells $1,2,\dots,\ell+5$ or $2,3,\dots,\ell+6$ depending on the parity of the first cell in the inherited gridding. Denote the omnibus encoding of this gridding of the body by $w_\pi$; this is a word over the alphabet $\{1,2,\dots,\ell+6\}$.

We now form a marked version of $w_\pi$. The lower tail of $\pi$ has length $t_{\text{SW}}^{\pi^\gridded} \geq 0$, while the upper tail has length $t_{\text{NE}}^{\pi^\gridded} \geq 0$. If $t_{\text{SW}}^{\pi^\gridded}$ (resp.~$t_{\text{SE}}^{\pi^\gridded}$) is non-zero, then there is a unique entry in the body which is adjacent to an entry of the lower (resp.~upper) tail. We mark the letter of $w_\pi$ which corresponds to this entry with an underline (resp.~overline), and denote the resulting marked version of $w_\pi$ by $\overline{w}_\pi$. The relative positions between all entries of the body and the two tails are now determined by $\overline{w}_\pi$, though the lengths of the tails are not captured in this word.

Let $\overline{\Sigma}$ be the extended alphabet consisting of the symbols $\{1,2,\dots,\ell+6\}$ together with over- and underlined versions of each. The discussion above defines an injective mapping from sum indecomposable permutations in $\C$ to $\overline{\Sigma}^\ast \times \mathbb{N} \times \mathbb{N}$ given by
\[
	\pi^\gridded \mapsto (\overline{w}_\pi, t_{\text{SW}}^{\pi}, t_{\text{NE}}^{\pi}).
\]
Define an ordering on $\overline{\Sigma}^\ast \times \mathbb{N} \times \mathbb{N}$ by taking product of the subword ordering on $\overline{\Sigma}^\ast$ and the usual orderings on the two copies of $\mathbb{N}$. Because $\overline{\Sigma}^\ast$ is well-quasi-ordered by Higman's Lemma and the product of well-quasi-orders is again well-quasi-ordered, $\overline{\Sigma}^\ast \times \mathbb{N} \times \mathbb{N}$ is well-quasi-ordered. Moreover, if $(\overline{w}_\sigma, t_{\text{SW}}^{\sigma}, t_{\text{NE}}^{\sigma}) \le (\overline{w}_\pi, t_{\text{SW}}^{\pi}, t_{\text{NE}}^{\pi})$ in this ordering then $\sigma\le\pi$ as the comparability on the first coordinate implies that the body of $\sigma$ embeds into the body of $\pi$ in a way preserving the relative positions of the entries adjacent to the two tails (a consequence of Observation~\ref{obs-order-pres}). The inequality of tail lengths then allows for the entire embedding of $\sigma$ into $\pi$ to be completed. Hence, with respect to subpermutation ordering, the sum indecomposable members of $\C$ are well-quasi-ordered, and so $\C$ is as well by Proposition~\ref{prop-wqo-sum-components}.
\end{proof}

We now turn to the second half of the argument---that all well-quasi-ordered subclasses of $\Av(321)$ are encoded by regular languages. Guided by Theorem~\ref{thm-wqo-fin-int}, we would like to check the involvement of sufficiently long members of $U$ in a subclass $\C$ by considering the encodings $(\eta_c\circ\omega)(\pi^\gridded)$ of greedy griddings of members of $\C$ and an appropriate value of $c$. To achieve this, we resort once more to the Dyck path encodings. First, as indicated in Figure~\ref{fig-antichains}, it is easy to see that the Dyck path encodings of members of $U$ form a regular language---outside of bounded prefixes and suffixes these words consist of repetitions of $\textsf{u}^2 \textsf{d}^2$.

In fact we are interested in the encodings of sets $U_{\ge q}$ for $q\in\mathbb{P}$, consisting of permutations in $U$ of length at least $q$. Noting that $U\setminus U_{\ge q}$ is finite for every value of $q$ we obtain the following.

\begin{proposition}
\label{prop-U-Dyck-regular}
For any positive integer $q$, the language of Dyck paths corresponding to the members of $U_{\geq q}$ is regular.
\end{proposition}

\begin{figure}
\begin{footnotesize}
\begin{center}
    \begin{tikzpicture}[scale=0.3, baseline=(current bounding box.center)]
		\plotpermbox{2}{2}{8}{8};
		\plotpartialperm{2/2,4/4,6/6,8/8};
		\plotpartialperm{5/1,7/3,9/5,11/7};
		\node [rotate=45] at (9.25,9.25) {{\scriptsize $\dots$}};
		\node [rotate=45] at (12.25,8.25) {{\scriptsize $\dots$}};
		\node [rotate=45] at (0.75,0.75) {{\scriptsize $\dots$}};
		\node [rotate=45] at (3.75,-0.25) {{\scriptsize $\dots$}};
    \end{tikzpicture} 
\quad\quad
    \begin{tikzpicture}[scale=0.3, baseline=(current bounding box.center)]
		\plotpermbox{2}{2}{8}{8};
		\plotpartialperm{2/2,4/4,6/6,8/8};
		\plotpartialperm{-1/3,1/5,3/7,5/9};
		\node [rotate=45] at (9.25,9.25) {{\scriptsize $\dots$}};
		\node [rotate=45] at (6.25,10.25) {{\scriptsize $\dots$}};
		\node [rotate=45] at (0.75,0.75) {{\scriptsize $\dots$}};
		\node [rotate=45] at (-2.25,1.75) {{\scriptsize $\dots$}};
    \end{tikzpicture} 
\end{center}
\end{footnotesize}
\caption{In any staircase gridding of an increasing oscillation, there can be at most three entries in a cell.}
\label{fig-3-entries-per-cell}
\end{figure}
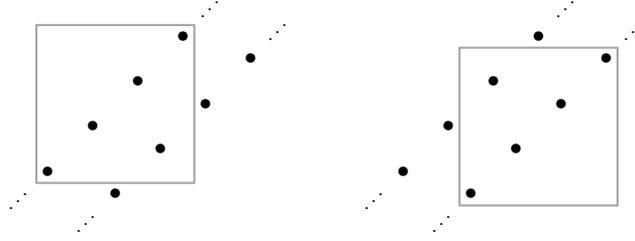

As demonstrated in Figure~\ref{fig-3-entries-per-cell} it is impossible for a cell of a staircase gridding of an increasing oscillation to contain four or more entries. As every member of $U$ is formed by adding two entries to an increasing oscillation, it follows that in every staircase gridding of a member of $U$ each cell may contain at most five entries. In particular, if an element $\mu\in U$ occurs as a subpermutation of $\pi \in \Av(321)$ with greedy gridding $\pi^\gridded$, and if we mark the letters of $\eta_c(\omega(\pi^\gridded))$ corresponding to any one copy of $\mu$ in $\pi$, no more than $5(c-1)$ occurrences of each letter will be marked by Observation \ref{obs-where-encoded}.

\begin{proposition}
\label{prop-Uq-regular}
Let $q$ be a positive integer and set $\W_q = \Av(321) \cap \Av(U_{\geq q})$. Further take $c$ to be any positive integer such that the omnibus encodings of all greedy staircase griddings of members of $\W_q$ are contained in $\G_c^\infty$ (such a value of $c$ is guaranteed to exist by Proposition~\ref{prop-avoid-shifts}). Then $\G^{\eta}_{c,\W_q}$, the language of panel encodings of greedy griddings of members of $\W_q$, is regular.
\end{proposition}

\begin{proof}
Given a panel encoding from $\G^{\eta}_c$, we seek to determine whether the corresponding permutation contains a member of $U$ of length $q$ or longer. Set $k=5(c-1)$. There is a non-deterministic transducer that takes words in $\G_{c}^{\eta}$ as input and outputs marked versions containing at most $k$ marked copies of each letter. Denote by $T$ the composition of this transducer, the transducer defined in Proposition~\ref{prop-panel-to-domino} (which outputs the domino encoding of the word formed by these marked letters), and the transducer defined in Proposition~\ref{prop-domino-to-Dyck} (which outputs the Dyck path corresponding to these marked letters). Thus given the encoding $\eta_c(\omega(\pi^\gridded))\in \G_{c}^{\eta}$ of a greedily gridded permutation $\pi^\gridded$, $T$ outputs the set of all Dyck paths corresponding to certain subpermutations of $\pi$. Crucially, our observations above show that for any member $\mu$ of the infinite antichain $U$, we have $\mu\le\pi$ if and only if $\mu\in T(\eta_c(\omega(\pi^\gridded)))$.

To conclude, we know from Proposition~\ref{prop-U-Dyck-regular} that the language, say $\mathcal{D}_{\geq q}$, of Dyck paths corresponding to the members of $U_{\geq q}$ is regular. Therefore we see that
\[
	\mathcal{G}_{c,\mathcal{W}_q}^\eta
	=
	\mathcal{G}_c^\eta\setminus T^{-1}(\D_{\geq q}),
\]
and thus $\mathcal{G}_{c,\mathcal{W}_q}^\eta$ is regular, as desired.
\end{proof}

We can now prove the second half of our main result.

\newenvironment{proof-of-321-rational-wqo}{\medskip\noindent {\it Proof of Theorem~\ref{thm-321-rational} (for well-quasi-ordered subclasses).\/}}{\qed\bigskip}
\begin{proof-of-321-rational-wqo}
Using Theorem \ref{thm-wqo-fin-int}, choose a positive integer $q$ such that $\C$ contains no element of $U_{\ge q}$, i.e., $\C \subseteq \W_q$, and choose $c$ according to Proposition~\ref{prop-avoid-shifts} so that the omnibus encodings all members of $\W_q$ are contained in $\G_c^{\infty}$. The minimal members of $\W_q \setminus \C$ form an antichain, say $B\subseteq\W_q$, which is finite because $\W_q$ is well-quasi-ordered. Thus we have
\[
	\G_{c,\C}^{\eta} = \G_{c, \W_q}^{\eta} \setminus \bigcup_{\beta \in B} \G_{c, \geq \beta}^{\eta}
\]
and, as all parts of the right hand side are known to be regular (by Propositions~\ref{prop-single-base-regular} and \ref{prop-Uq-regular}) and $B$ is finite, we may conclude that $\G_{c,\C}^{\eta}$ is regular. It follows that the generating function for $\G_{c,\C}^{\eta}$, which is equal to that of $\C$, is rational.
\end{proof-of-321-rational-wqo}

\section{Conclusion}

While we opened the paper by emphasising the differences between the two Catalan permutation classes defined by avoiding $312$ and $321$, respectively, our main result shows that they do share a remarkable property. Every finitely based or well-quasi-ordered proper subclass of either of these classes has a rational generating function. Of course, stating the result in this way obscures a serious difference: \emph{all} subclasses of the $312$-avoiding permutations are \emph{both} finitely based and well-quasi-ordered.

One interested in actually computing these generating functions will notice an even more striking difference. While computing the enumeration of subclasses of $312$-avoiding permutations is essentially trivial (as outlined in~\cite{albert:simple-permutat:}), for subclasses of $321$-avoiding permutations the enumeration method we have presented appears to be impractical.

Another context in which the differences between these classes are readily apparent is that of Wilf-equivalence. Two permutation classes $\C$ and $\D$ are said to be \emph{Wilf-equivalent} if they are equinumerous, i.e., $|\C_n|=|\D_n|$ for all $n$. For classes defined by avoiding $312$ and a single additional restriction, Albert and Bouvel~\cite{Albert:A-general-theor:} have provided a conjecturally complete classification of the Wilf-equivalences. However, while there are some enumerative coincidences among classes defined by avoiding $321$ and a single additional restriction, empirically there does not appear to be anywhere near the same amount of collapse (into a small number of Wilf-equivalence classes). A related result was proved by Albert, Atkinson, Brignall, Ru\v{s}kuc, Smith, and West~\cite{albert:growth-rates-fo:}, who gave some sufficient conditions for the classes of $\{321,\alpha\}$- and $\{321,\beta\}$-avoiding permutations to have the same exponential growth rate.

We believe that the techniques introduced in this work---especially the panel encoding of Section~\ref{sec-panel-encoding}---will find many more applications. To introduce these we first observe that in the language of geometric grid classes~\cite{albert:inflations-of-g:,albert:geometric-grid-:,bevan:growth-rates-of:geom}, the $321$-avoiding permutations form the grid class of the infinite matrix
\[
	\fnmatrix{lllll}{%
	&&&\reflectbox{$\ddots$}&\reflectbox{$\ddots$}\\
	&&1&1\\
	&1&1\\
	1&1}.
\]
This is equivalent to the observation, made at the end of Section~\ref{sec-staircase}, that the $321$-avoiding permutations are precisely those that can be drawn on two parallel rays (see the first picture in Figure~\ref{fig-infinite-geom}). While a great deal is known about geometric grid classes, the present work can be viewed as an initial attempt to extend that theory to infinite matrices (another initial attempt in this direction is \cite{albert:enumerating-ind:}). One aspect of the infinite geometric grid class view of $321$-avoiding permutations that seems particularly important is that the cells can be labelled so that cell $i$ interacts only with cells $i-1$ and $i+1$, in the sense that the relative positions and values of any two entries in cells whose indices differ by more than one depend only on the indices of the cells, giving the class a ``path-like'' structure.

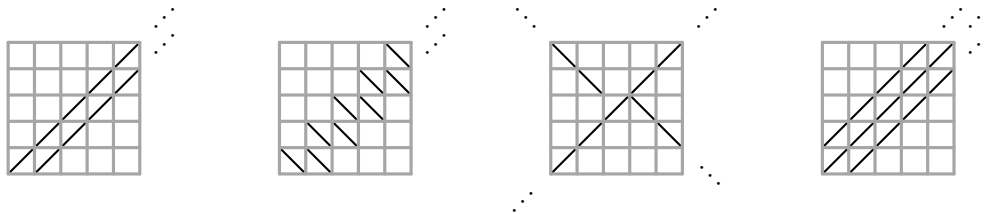
\begin{figure}
\begin{center}
	\begin{tikzpicture}[scale=0.35, baseline=(current bounding box.center)]
		\draw [thick, line cap=round] (3,3)--(8,8);
		\draw [thick, line cap=round] (4,3)--(8,7);
		\node[rotate=45] at (3,2) {\phantom{$\dots$}};
		\node[rotate=45] at (2,2) {\phantom{$\dots$}};
		\node[rotate=45] at (9,9) {$\dots$};
		\node[rotate=45] at (9,8) {$\dots$};
		\foreach \i in {3,4,5,6,7,8}{
			\draw [darkgray, very thick, rounded corners=0.01, line cap=round] (3,\i)--(8,\i);
			\draw [darkgray, very thick, rounded corners=0.01, line cap=round] (\i,3)--(\i,8);
		}
	\end{tikzpicture}
\quad
	\begin{tikzpicture}[scale=0.35, baseline=(current bounding box.center)]
		\foreach \i in {3,4,5,6}{
			\draw [thick, line cap=round] (\i,{\i+1})--({\i+1},\i);
			\draw [thick, line cap=round] ({\i+1},{\i+1})--({\i+2},\i);
		}
		\draw [thick, line cap=round] (7,8)--(8,7); 
		\node[rotate=45] at (3,2) {\phantom{$\dots$}};
		\node[rotate=45] at (2,2) {\phantom{$\dots$}};
		\node[rotate=45] at (9,9) {$\dots$};
		\node[rotate=45] at (9,8) {$\dots$};
		\foreach \i in {3,4,5,6,7,8}{
			\draw [darkgray, very thick, rounded corners=0.01, line cap=round] (3,\i)--(8,\i);
			\draw [darkgray, very thick, rounded corners=0.01, line cap=round] (\i,3)--(\i,8);
		}
	\end{tikzpicture}
\quad
	\begin{tikzpicture}[scale=0.35, baseline=(current bounding box.center)]
		\draw [thick, line cap=round] (3,3)--(8,8);
		\draw [thick, line cap=round] (3,8)--(5,6);
		\draw [thick, line cap=round] (6,6)--(8,4);
		\node[rotate=45] at (2,2) {$\dots$};
		\node[rotate=135] at (2,9) {$\dots$};
		\node[rotate=135] at (9,3) {$\dots$};
		\node[rotate=45] at (9,9) {$\dots$};
		\foreach \i in {3,4,5,6,7,8}{
			\draw [darkgray, very thick, rounded corners=0.01, line cap=round] (3,\i)--(8,\i);
			\draw [darkgray, very thick, rounded corners=0.01, line cap=round] (\i,3)--(\i,8);
		}
	\end{tikzpicture}
\quad
	\begin{tikzpicture}[scale=0.35, baseline=(current bounding box.center)]
		\draw [thick, line cap=round] (3,4)--(7,8);
		\draw [thick, line cap=round] (3,3)--(8,8);
		\draw [thick, line cap=round] (4,3)--(8,7);
		\node[rotate=45] at (3,2) {\phantom{$\dots$}};
		\node[rotate=45] at (2,2) {\phantom{$\dots$}};
		\node[rotate=45] at (8,9) {$\dots$};
		\node[rotate=45] at (9,9) {$\dots$};
		\node[rotate=45] at (9,8) {$\dots$};
		\foreach \i in {3,4,5,6,7,8}{
			\draw [darkgray, very thick, rounded corners=0.01, line cap=round] (3,\i)--(8,\i);
			\draw [darkgray, very thick, rounded corners=0.01, line cap=round] (\i,3)--(\i,8);
		}
	\end{tikzpicture}
\end{center}
\caption{The $321$-avoiding staircase, the negative staircase, an infinite spiral, and a thickened staircase.}
\label{fig-infinite-geom}
\end{figure}

It would therefore be natural to attempt to extend the results established here to other infinite geometric grid classes possessing a similar structure. Two more examples are given by the second and third pictures shown in Figure~\ref{fig-infinite-geom}.

The class corresponding to the second picture of Figure~\ref{fig-infinite-geom}, which we call the \emph{negative staircase}, demonstrates one reason why our techniques cannot be translated automatically to all path-like geometric grid classes. Indeed, while greedy staircase griddings are easy to describe for the $321$-avoiding staircase, the issue is not so clear-cut for the negative staircase. To see this, consider the permutations $4123$ and $2341$. Both of these permutations can be drawn on the negative staircase, as demonstrated below.
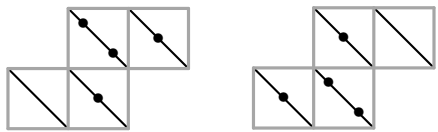
\begin{center}
    \begin{tikzpicture}[scale=0.2, baseline=(current bounding box.center)]
		\draw [thick, line cap=round] (0,4)--(4,0);
		\draw [thick, line cap=round] (4,4)--(8,0);
		\draw [thick, line cap=round] (4,8)--(8,4);
		\draw [thick, line cap=round] (8,8)--(12,4);
		\draw [darkgray, very thick, rounded corners=0.01, line cap=round] (0,0)--(8,0);
		\draw [darkgray, very thick, rounded corners=0.01, line cap=round] (0,4)--(12,4);
		\draw [darkgray, very thick, rounded corners=0.01, line cap=round] (4,8)--(12,8);
		\draw [darkgray, very thick, rounded corners=0.01, line cap=round] (0,0)--(0,4);
		\draw [darkgray, very thick, rounded corners=0.01, line cap=round] (4,0)--(4,8);
		\draw [darkgray, very thick, rounded corners=0.01, line cap=round] (8,0)--(8,8);
		\draw [darkgray, very thick, rounded corners=0.01, line cap=round] (12,4)--(12,8);
		\plotpartialperm{6/2,5/7,7/5,10/6};
    \end{tikzpicture} 
\quad\quad
    \begin{tikzpicture}[scale=0.2, baseline=(current bounding box.center)]
		\draw [thick, line cap=round] (0,4)--(4,0);
		\draw [thick, line cap=round] (4,4)--(8,0);
		\draw [thick, line cap=round] (4,8)--(8,4);
		\draw [thick, line cap=round] (8,8)--(12,4);
		\draw [darkgray, very thick, rounded corners=0.01, line cap=round] (0,0)--(8,0);
		\draw [darkgray, very thick, rounded corners=0.01, line cap=round] (0,4)--(12,4);
		\draw [darkgray, very thick, rounded corners=0.01, line cap=round] (4,8)--(12,8);
		\draw [darkgray, very thick, rounded corners=0.01, line cap=round] (0,0)--(0,4);
		\draw [darkgray, very thick, rounded corners=0.01, line cap=round] (4,0)--(4,8);
		\draw [darkgray, very thick, rounded corners=0.01, line cap=round] (8,0)--(8,8);
		\draw [darkgray, very thick, rounded corners=0.01, line cap=round] (12,4)--(12,8);
		\plotpartialperm{2/2,5/3,7/1,6/6};
    \end{tikzpicture} 
\end{center}
Moreover, up to shifting the choice of cells, the griddings shown above are the only negative staircase griddings of $4123$ and $2341$. The permutation $4123$ shows that we cannot take the members of the first cell to consist of the maximum initial decreasing subsequence. On the other hand, $2341$ shows that we cannot define greedy staircase griddings by the value either. Thus any definition of greedy negative staircase griddings would have to incorporate at least a slightly more global sense of the permutation to be gridded than was required for the $321$-avoiding staircase.

In dealing with either the negative staircase class or the \emph{infinite spiral class} (the third picture in Figure~\ref{fig-infinite-geom}), one would also have to develop a replacement for the Dyck path encoding. However, we do not believe this step is, in and of itself, a major impediment, as the role of the Dyck path encoding is just a proxy for maintaining a set of requirements in finitely many states, and it seems clear that similar devices could be developed for other classes obtained from regular path-like structures.

Much more serious issues present themselves if we remove the path-like condition on the occupied cells; for instance, consider the class of permutations that can be drawn on the \emph{thickened staircase} shown on the far right of Figure~\ref{fig-infinite-geom}. This class is a proper subclass of the $4321$-avoiding permutations and so to see that we cannot hope for a result like Theorem~\ref{thm-321-rational} in this context we need only note that this class contains the class of $321$-avoiding permutations. On the language level, even if we could define the domino encoding in this setting, we could not impose the small ascent condition on the encodings of words describing members of this class, so their encodings would not lie in $\L^{\infty}$, and thus the panel encoding could not be applied.

Finally, an emerging topic of interest in the general study of permutation classes has been strong and broad rationality and algebraicity (see~\cite{albert:inflations-of-g:,albert:subclasses-of-t:}). While the presence of infinite antichains necessarily implies that a class has subclasses whose generating functions are not D-finite, we have shown that certain subclasses of the $321$-avoiding permutations are nevertheless well-structured. To make this notion precise we say that a class is \emph{broadly rational} if it and all of its finitely based subclasses have rational generating functions and/or \emph{strongly rational} if this holds for \emph{all} of its subclasses. Therefore Theorem~\ref{thm-321-rational} shows that all proper subclasses of the $321$-avoiding permutations are broadly rational. The same counting argument as above shows that every strongly rational class must be well-quasi-ordered. Thus Theorem~\ref{thm-321-rational} also implies the following.

\begin{corollary}
\label{cor-321-strong-rat}
A subclass of $\Av(321)$ is strongly rational if and only if it is well-quasi-ordered.
\end{corollary}

This represents one more piece of evidence for the following conjecture (which is also supported by the results of~\cite{albert:subclasses-of-t:}).

\begin{conjecture}
A permutation class is strongly rational if and only if it is well-quasi-ordered and does not contain the class of $312$-avoiding permutations or any symmetry of it.	
\end{conjecture}

\noindent{\bf Acknowledgements.} Significant inspiration for this research came from the work of Lozin~\cite{lozin:minimal-classes:}, who proved that while the class of bipartite inversion graphs (the inversion graphs of $321$-avoiding permutations) has unbounded clique-width, every proper subclass of this class has bounded clique-width. We are also grateful to Michael Engen, Jay Pantone, and the referees for their numerous suggestions and corrections.

\bibliographystyle{acm}
\bibliography{av321-refs}

\end{document}